\documentclass[12pt,amstex]{amsart}

\usepackage{mathptmx}
\usepackage{mathrsfs}
\usepackage{verbatim}
\usepackage{url}
\usepackage[all]{xy}
\usepackage{color}

\usepackage[colorlinks=true,citecolor=blue]{hyperref}
\usepackage{stmaryrd}
\usepackage{epsfig}
\usepackage{amsmath,bm}
\usepackage{amssymb}
\usepackage{amssymb,amsthm,amsmath}
\usepackage{mathrsfs}
\usepackage{amscd}
\usepackage{graphicx}
\usepackage{pstricks}
\usepackage{theoremref}

\newtheorem{thm}{Theorem}[section]
\newtheorem{lem}[thm]{Lemma}

\newtheorem{cor}[thm]{Corollary}
\theoremstyle{definition}
\newtheorem{de}[thm]{Definition}

\theoremstyle{remark}
\newtheorem{rem}[thm]{Remark}

\numberwithin{equation}{section}
\def \N {\mathbb{N}}

\def \Z {\mathbb{Z}}

\def \O {\mathcal{O}}

\def \G {\mathcal{G}}

\def \E {\mathcal E}

\def \RP {{\bf RP}}

\def \id {{\rm id}}

\def \ep {\epsilon}
\def \d {\delta}

\def \lra {\longrightarrow}
\def \ra {\rightarrow}

\def \RP {\textbf{RP}}

\begin{document}

\title{On structure theorems and non-saturated examples}
	
\author{Qinqi Wu}
\author{Hui Xu}
\author{Xiangdong Ye}

\address{Department of Mathematics, University of Science and Technology of China, Hefei, Anhui, 230026, P.R. China}

\email{qinqiwu@mail.ustc.edu.cn}
\email{huixu2734@ustc.edu.cn}
\email{yexd@ustc.edu.cn}


\thanks{The third author is supported by NNSF 12031019}

\date{}
\begin{abstract}
For any minimal system $(X,T)$ and $d\geq 1$ there is an associated minimal system $(N_{d}(X), \G_{d}(T))$,
where $\G_{d}(T)$ is the group generated by $T\times\cdots\times T$ and $T\times T^2\times\cdots\times T^{d}$
and $N_{d}(X)$ is the orbit closure of the diagonal under $\G_{d}(T)$.  It is known that
the maximal $d$-step pro-nilfactor of $N_d(X)$ is $N_d(X_d)$, where $X_d$ is the maximal $d$-step pro-nilfactor of $X$.

In this paper, we further study the structure of $N_d(X)$.
We show that the maximal distal factor of $N_d(X)$ is $N_d(X_{dis})$ with $X_{dis}$ being the maximal distal factor
of $X$, 
and prove that  as minimal systems $(N_{d}(X), \G_{d}(T))$ has the same structure theorem as $(X,T)$.
In addition, a non-saturated metric example $(X,T)$ is constructed, which  is not $T\times T^2$-saturated
and is a Toeplitz minimal system.
\end{abstract}

\maketitle
\markboth{On structure theorems and non-saturated examples}{Wu, Xu, Ye}

\section{Introduction}	

In the 1970's, Furstenberg gave a dynamical proof of the Szemeredi's theorem via establishing a version of
multiple ergodic recurrence  theorem ({\bf MERT} for short), which states that in any measure preserving
system $(X,\mathcal{X}, \mu, T)$, for any $d\in\N$ and $A\in\mathcal{X}$ with $\mu(A)>0$
there always exists $n\in\N$ such that
$\mu(A\cap T^{-n}A\cap\cdots\cap T^{-dn}A)>0$.
A counterpart in topological dynamics is the topological multiple recurrence theorem
({\bf TMRT} for short), which sates that for any nonempty open set $A$ in a minimal
topological dynamical system $(X,T)$ and $d\in \N$, there always exists
$n\in\N$ such that $A\cap T^{-n}A\cap\cdots\cap T^{-dn}A\neq\emptyset$. An easy
implication of {\bf TMRT} is the classical van der Waerden's theorem on arithmetic progressions.

The multiple recurrence theorems both in ergodic theory and topological dynamics are focusing on studying
the orbit of $(x,x,\cdots, x)$ under $\tau_{d}(T)=T\times \cdots\times T^{d}$ in the product space
$X\times \cdots\times X$. In ergodic theory, this can be captured by the limits of the following
multiple ergodic averages
\begin{equation}\label{eq1.1}
\frac{1}{N}\sum_{k=0}^{N-1}f_1(T^kx)f_2(T^{2k}x)\cdots f_{d}(T^{dk}x),
\end{equation}
the convergence in $L^2(\mu)$ of which has been finally solved in \cite{HK, Zieg}  after nearly 30 years'
efforts by many mathematicians.  In order to study the averages in (\ref{eq1.1}), the characteristic
factors were introduced to capture the limits of (\ref{eq1.1}), which were finally shown to be an efficient way.

In \cite{G94}, Glasner established  counterparts in topological dynamics which are called  topological
characteristic factors. Given a factor map $\pi:(X,T)\ra (Y,T)$ between minimal topological dynamical
systems and $d\geq 2$,  $(Y,T)$ is said to be a {\it $d$-step topological characteristic factor} of
$(X,T)$ if there is a dense $G_{\delta}$ subset $\Omega$ of $X$ such that for any $x\in \Omega$, the orbit
closure $L_x=\overline{\O}((x,\cdots,x), \tau_{d}(T))$ is $\pi^{(d)}:=\pi\times\cdots\times\pi$-{\it saturated},
which means that $(\pi^{(d)})^{-1}(\pi^{(d)}L_x)=L_x$. Furthermore, a closely related system $N_d(X)$ under
the action of the group $\G_d(T):=\langle \sigma_d(T),\tau_d(T)\rangle$ generated by
$\sigma_d(T):=T\times\cdots\times T$ ($d$-times) and $\tau_d(T)$ was also studied in \cite{G94},
where $N_d(X)$ is the orbit closure of $(x,\cdots,x)$ under $\G_d(T)$. It was shown in \cite{G94}
that $(N_d(X), \G_d(T))$ is also minimal for a minimal topological dynamical system $(X,T)$.

In \cite{GHSWY},  the authors further studied the topological characteristic factors and the dynamics of $(N_d(X), \G_d(T))$.
Up to an almost  1-1 modification, they showed that the topological characteristic factors are the pro-nilfactors
(see \cite[Theorem A]{GHSWY}), which are the analogies in the ergodic situation. For other study on topological characteristic factors,
see \cite{CS-1,CS-2}.

Another interesting and useful
result they showed  on the dynamics of $(N_d(X), \G_d(T))$ is the following theorem \cite[Theorem C]{GHSWY}.

\medskip

\noindent {\bf Theorem A}:\label{thmA}
{\it Let $(X,T)$ be minimal and $n\in\N$. Then $(X,T^n)$ is minimal if and only if $N_d(T^n)=N_d(T)$ for any $d\in\N$.
}

\medskip

\noindent By this theorem, they showed that if $(X,T^2)$ is minimal, then {\bf TMRT} holds along odd numbers
(actually they showed more, see \cite[Theorem D]{GHSWY}) and  for a totally minimal system $(X,T)$ and an
integral quadratic polynomial $p(n)=an^2+bn+c$ with $a\neq0$, there is a dense $G_{\delta}$ subset $\Omega$
of $X$ such that for any $x\in \Omega$, $\{T^{p(n)}x: n\in\Z\}$ is dense in $X$ (see \cite[Theorem E]{GHSWY}).

According to their powerful structure theorem of topological characteristic factors, they determined the maximal
equicontinuous factors of $N_d(X)$ (\cite[Theorem B]{GHSWY}), which says that the maximal equicontinuous factor
of $N_d(X)$ is the very $N_d(X_{eq})$, where $X_{eq}$ is  the maximal equicontinuous factor of $X$. This is also
crucial in the proof of Theorem A. Note that recently, Lian and Qiu \cite{LQ} showed
that for any $k\in\N$, the maximal $k$-step nilfactor
of $N_d(X)$ is $N_d(X_k)$, where $X_k$ is the maximal $k$-step pro-nilfactor of $X$.

\medskip
In this paper, we will give another approach to Theorem A. That is, starting from the fact that the maximal equicontinuous factor
of $N_d(X)$ is $N_d(X_{eq})$, we use an equivalence condition for $N_d(T)=N_d(T^n)$ (Theorem \ref{equivalence}) to the get the conclusion,
instead of converting to the equicontinuous cases as did in \cite{GHSWY}.
\medskip

Based on the results obtained in \cite{GHSWY, LQ} one naturally expect that the maximal distal factor of $N_d(X)$ should be $N_d(X_{dis})$, where $X_{dis}$ is  the maximal distal factor of $X$. We show that indeed it is the case.

\medskip
\noindent {\bf Theorem B}:\label{thmB}
{\it Let $(X,T)$ be a minimal topological dynamical system with $X_{dis}$ being its maximal distal factor.
Then for any $d\geq 2$, the maximal distal factor of $N_d(X)$ is $N_d(X_{dis})$.}

\medskip

A powerful tool in studies of minimal systems is the structure theorem. In 1963, Furstenberg in \cite{Fur63}
established the structure theorem for minimal distal systems.
Then the structure theorem for pointed distal minimal systems was
established by Veech in \cite{Veech}.
Finally, Ellis, Glasner and Shapiro in \cite{EGS75}, McMahon \cite{Mc78}, Veech \cite{V83} gave the structure theorem for general minimal systems.  In this paper, we will show the following 
result for $N_d(X)$ as a minimal system associated to $(X,T)$.

\medskip
\noindent {\bf Theorem C}:\label{thmC}
{\it Let $(X,T)$ be a minimal topological dynamical system and $d\in \N$. Then $(N_d(X), \G_d(T))$ has the same
structure theorem as $(X,T)$. In particular, if $(X,T)$ is distal (resp. {\bf PI, HPI}), then so is $(N_d(X), \G_d(T))$.}

\medskip
The key ingredient to show the minimality of $N_d$ under the $\G_d$ action is that if $(x_1,\ldots,x_d)$ is minimal for $\sigma_d$
then it is minimal for $\G_d$. To show Theorem C, we need a generalization of this fact. Namely, let $n, d\in\N$. For $(x_i^1)_{1}^d, \ldots, (x_i^n)_{1}^d\in N_d(X)$, if $((x_i^1)_{1}^d, \ldots, (x_i^n)_{1}^d)$ is a minimal point in $(X^{n}, T)$ under the diagonal action, then
$ ((x_i^1)_{1}^d, \ldots, (x_i^n)_{1}^d)$ is a minimal point in $ ((N_{d}(X))^{n}, \G_{d}(T))$ under the diagonal action, see Lemma \ref{appi-Nd}.

\medskip
As pointed in \cite{GHSWY}, the almost 1-1 modification in their structure theorem for topological characteristic
factors is necessary due to an example given in \cite{G94}, which describes an almost automorphic  system whose
maximal equicontinuous factor is not its characteristic factor. Unfortunately,  the example given in \cite{G94}
is not a metriziable system. Here we will construct a metric minimal system which is an almost 1-1 extension
of its maximal equicontinuous factor such that the factor map is not saturated. We state it as follows.

\medskip
\noindent {\bf Theorem D}:\label{thmD}
{\it There exists a minimal topological dynamical system $(X,T)$   such that
\begin{enumerate}
\item  $\pi: (X,T)\ra (X_{eq}, T)$ is an almost 1-1 extension, where $X_{eq}$ is its maximal equicontinuous factor;
\item $(X_{eq}, T)$ is not a topological characteristic factor of $(X,T)$, i.e. there does not exist a dense $G_{\delta}$ subset $\Omega$ of $X$ such that for any $x\in \Omega$, $L_x=\overline{\O}((x,x), T\times T^2)$ is $\pi\times\pi$-saturated.
\end{enumerate}
}
\medskip
Actually, the system we construct is an irregular Toeplitz system.

\medskip
\noindent{\bf Organization of the paper}. In Section 2, we give some preliminaries. In Section 3, we give
a proof of Theorem A. We prove Theorems B and  C in Sections 4 and  5 respectively. In Section 6 we
present some saturated examples and finally we prove Theorem D in Section 7 via constructing a desired Toeplitz sequence.

\section{Preliminaries}

In this section we provide notions and lemmas needed for our proofs.
\subsection{General topological dynamical systems }
By a {\it topological dynamical system} (t.d.s. or system for short) we mean a pair $(X, T)$ with $X$ being a
compact Hausdorff space and $T$ being a homeomorphism of $X$. More generally, by  a  topological dynamical system we
mean a triple $(X, G, \phi)$ with a compact metric space $X$, a topological group $G$ and a homomorphism
$\phi: G\ra {\rm Homeo}(X)$, where ${\rm Homeo}(X)$ denotes the group of homeomorphisms of $X$.
For brevity, we usually use $(X,G)$ to denote $(X,G,\phi)$and use
$gx$ or $g(x)$ instead of $(\phi(g))(x)$ for $g\in G$ and $x\in X$.  Note that the system $(X,T)$ corresponds
to the case of $G$ being $\Z$.

In this paper, we always assume that $X$ is a compact metric space with a metric $\rho$. Let $(X,G)$ be a t.d.s.
For $x\in X$,  the orbit of $x$ under the action of $G$ is defined to be $\{gx: g\in G\}$, which is denoted by
$Gx$ or $\O(x, G)$ (or $\O(x, T)$ in case of $G=\Z$). Correspondingly, $\overline{Gx}$ or $\overline{\O}(x, G)$
will denote the orbit closure of $x$.  Similarly, for a subset $A\subset X$,  the orbit of $A$ is
$\{gx: x\in A, g\in G\}$, which is denoted by $G(A)$ or $\O(A,G)$.

 A subset $A\subset X$ is called {\it invariant} under the action of $G$ or $G$-invariant if $G(A)=A$.  A t.d.s. $(X,G)$ is said to be  {\it minimal} if there is no nonempty proper $G$-invariant closed subset.  Note that $(X,G)$ is minimal if and only if the orbit closure of every point is dense in  $X$.  If there is a point of $X$ whose orbit is dense in $X$ then we say that the system is {\it topologically transitive}.  Furthermore, if the product system $(X\times X, T)$ under the diagonal action (i.e. $T(x,y)=(Tx,Ty)$) is topologically transitive, then we say that $(X, T)$ is {\it weakly mixing}.

A {\it factor map}  between the t.d.s. $(X,G)$ and $(Y,G)$ is a continuous surjective
map $\pi: X\ra Y$ which intertwines the actions, i.e., $\pi(gx)=g\pi(x)$ for any $g\in G$ and $x\in X$. In this case, we say that $(Y,G)$ is a {\it factor} of $(X,G)$ or $(X,G)$ is an {\it extension} of $(Y,G)$. Let $\pi: (X, T)\ra (Y, T)$ be a factor map. Then $R_{\pi}=\{(x,y)\in X\times X: \pi(x)=\pi(y)\}$ is a closed $G$-invariant equivalence relation and $Y\cong X/R_{\pi}$.

Let $X$ be a set and $d\in\N$ with $d\geq 2$.  The diagonal of $X^d$ is defined to be
$$\Delta_{d}(X)=\{(x,x,\cdots,x): x\in X\}.$$
We use $\Delta(X)$ to denote $\Delta_2(X)$ for short. For $x\in X$, let $x^{(d)}$ denote $(x,x,\cdots,x)$ ($d$-times) for short.

\subsection{Distal, proximal and regionally proximal relations }
Let $(X, G)$ be a t.d.s.. Given $(x,y)\in X^2$, it is a {\it proximal} pair if
$\liminf_{g\in G}\rho(gx,gy)= 0$; it is a {\it distal} pair if it is not proximal.
Denote by ${\bf P}(X,G)$ the set of proximal pairs of $(X,G)$. It is also called the
{\it proximal relation}.   Note that the proximal relation is $G$-invariant but not
an equivalence 	relation in general. However, when ${\bf P}(X,G)$ is closed in $X\times X$,
it is an equivalence relation (see \cite[Chapter 6, Corollary 11]{Aus}). $(x,y)\in X^2$ is a
{\it regionally proximal} pair if there exist sequences $(x_n)$, $(y_n)$ of $X$ with
$x_n\ra x, y_n\ra y$ and a sequence $(g_n)$ of $G$ such that $\rho(g_nx_n, g_ny_n)\ra 0$.
The set of regionally proximal pair of $(X,G)$ is denoted by ${\bf RP}(X,G)$, which is
a closed invariant relation but not an equivalence relation in general. However,
when $G$ preserves an invariant probability measure ${\bf RP}(X,G)$ is an equivalence relation (see \cite[Chapter 9]{Aus}).

A t.d.s. $(X,G)$ is {\it distal} if ${\bf P}(X,G)=\Delta(X)$, and is {\it equicontinuous} if for every $\epsilon>0$
there exists $\delta>0$ such that $\rho(x,y)<\delta$ implies $\rho(gx,gy)<\epsilon$ for every $g\in G$.
Any equicontinuous system is distal. It is well known that every t.d.s. has a maximal equicontinuous factor,
and a maximal distal factor respectively. Let $S_{eq}$ (resp. $S_{dis}$) be the smallest closed
invariant equivalence relation on $X$ for which the factor $X/S_{eq}$ (resp. $X/S_{dis}$) is an equicontinuous
(resp. distal) system. It was shown in \cite{EG} that $S_{eq}$ (resp. $S_{dis}$) is the smallest closed invariant equivalence relation containing the regionally proximal relation (resp. proximal relation).

In \cite{JandG02}, Auslander and Glasner defined the capturing operation,
a kind of reverse orbit closure, to characterize the equicontinuous and distal relations in minimal systems.
If $A\subset X$, the capturing set of $A$ is defined by $$C(A)= \{x\in X:\overline{\O}(x,G) \cap A\ne\varnothing \}.$$
 We say that $A$ is {\it a capturing set} if $C(A)= A$.
For a symmetric and reflexive relation $R$,  set $\E({R})= \bigcup\{R^{n}: n=1,2,\ldots \}$, where
\[ R^{n}=\{(x,z): \exists y_1,\cdots, y_{n-1} ~~s.t.~~ (x,y_1), (y_1,y_2),\cdots, (y_{n-2},y_{n-1}),(y_{n-1},z)\in R\}.\]
In \cite{JandG02}, the authors showed that  $S_{dis}(X)$ and $S_{eq}(X)$ are capturing sets in $(X\times X,G)$ and
\begin{thm}\cite[Theorem 4.5]{JandG02}\label{JG}
Let $(X,G)$ be a minimal system. Then $S_{dis}(X)= C(\overline{\E({\overline{{\bf P}}})})$.
\end{thm}
\begin{lem}\cite[Theorem 2.1]{JandG02}\label{liftofdis}
Let $\pi: (X,T)\rightarrow (Y,T)$ be an extension between minimal systems.  Then $\pi(S_{dis}(X))=S_{dis}(Y)$.
\end{lem}
An extension $\pi: (X,G)\ra (Y,G)$ is called {\it proximal} if $R_{\pi}\subset {\bf P}(X)$; $\pi$ is {\it distal}
if ${\bf P}(X)\cap R_{\pi}=\Delta(X)$; $\pi$ is {\it equicontinuous} if for any $\varepsilon>0$ there is $\delta>0$
such that for any $(x_1,x_2)\in R_{\pi}$ with $\rho(x_1,x_2)<\delta$, we have $\sup_{g\in G}\rho(gx_1,gx_2)<\varepsilon$.
An extension $\pi: (X,G)\ra (Y,G)$ between minimal systems is called almost $1-1$ if the set $\{y\in Y: |\pi^{-1}(y)|=1\}$ is a dense $G_{\delta}$ subset of $Y$, which is equivalent to the existence of a point in $Y$ whose fibre is a singleton.

\begin{lem}\label{finite index subgroup}
If $\pi: (X,G)\lra (Y,G)$ is a proximal (resp. equicontinuous, distal) extension and $H$ is a subgroup of $G$ of finite index, then
the induced extension $\pi_{H}: (X,H)\lra (Y,H)$ is also  proximal (resp. equicontinuous, distal).
\end{lem}
\begin{proof}
Suppose $x_1,x_2\in \pi^{-1}(y)$ for some $y\in Y$.  Let $G=g_1H\cup\cdots\cup g_n H$ for some $g_1,\cdots,g_n\in G$.

If $\inf_{g\in G}\rho(gx,gy)>0$, then $\inf_{h\in H}\rho(hx,hy)>0$.
Thus if $\pi: (X,G)\lra (Y,G)$ is a distal extension then so is $\pi_{H}: (X,H)\lra (Y,H)$.

If  $\inf_{h\in H}\rho(hx,hy)>0$, then $\inf_{h\in H}\rho(g_ihx,g_ihy)>0$ for any $i\in\{1,\cdots,n\}$.
Thus $\inf_{g\in G}\rho(gx,gy)>0$. Therefore, if $\pi: (X,G)\lra (Y,G)$ is a proximal extension
then so is $\pi_{H}: (X,H)\lra (Y,H)$.

If $\pi: (X,G)\lra (Y,G)$ is an equicontinuous extension, then  for any $\varepsilon>0$, there is
some $\delta>0$ such that for any $(x_1,x_2)\in R_{\pi}$ with $\rho(x_1,x_2)\leq\delta$, we have
$\sup_{g\in G}\rho(gx_1,gx_2)<\varepsilon$. Then it follows trivially that $\sup_{g\in H}\rho(gx_1,gx_2)<\varepsilon$.
Thus $\pi: (X,H)\lra (Y,H)$ is also an equicontinuous extension.
\end{proof}

\begin{lem}\label{extension for product}
If $\pi_{i}:(X_i, G)\lra (Y_i, G)$ is a proximal (resp. equicontinuous, distal) extension for each $i\in\{1,\cdots,n\}$, then the extension  $\prod_{i=1}^n \pi_i: (\prod_{i=1}^nX_i, G)\lra (\prod_{i=1}^{n}Y_i, G)$ of the product system is also proximal (resp. equicontinuous, distal).
\end{lem}
\begin{proof}
Suppose that each $\pi_i$ is a proximal extension. We  show that  $\prod_{i=1}^n \pi_i$ is also a proximal extension just for $n=2$. The general case follows by induction on $n$.

Let $(x_1,x_2),(x_1',x_2')\in R_{\pi_1\times\pi_2}$. Then there is a sequence $(g_n)$ in $G$ such that $g_nx_1\rightarrow z_1$ and $g_nx_1'\rightarrow z_1$ for some $z_1\in X_1$. By taking some subsequence, we may assume that $g_nx_2\rightarrow z_2$ and $g_nx_2'\rightarrow z_2'$ for some $z_2,z_2'\in X_2$. Then $(z_2,z_2')\in R_{\pi_2}$. So there is some sequence $(h_n)$ in $G$ such that $h_nz_2\rightarrow z_3$ and $h_nz_2'\rightarrow z_3$ for some $z_3\in X_2$. We may further assume that $h_nz_1\rightarrow z_0$ for some $z_0\in X_1$.  For each $k\geq 1$, there is some $n_k\geq 1$ such that for any $n\geq n_k$,
$$\rho(h_kg_{n}x_1, h_kz_1), ~~\rho(h_kg_{n}x_1', h_kz_1), ~~\rho(h_kg_{n}x_2, h_kz_2), ~~\rho(h_kg_{n}x_2', h_kz_2')<\frac{1}{k}.$$
Then $h_kg_{n_{k}}(x_1,x_2)\rightarrow (z_0, z_3)$ and  $h_kg_{n_{k}}(x_1',x_2')\rightarrow (z_0, z_3)$. Hence $((x_1,x_2),(x_1',x_2'))\in {\bf P}(X_1\times X_2)$ and it follows that $\pi_1\times\pi_2$ is also proximal.

\medskip
It is trivial by definitions for equicontinuous and distal extensions.
\end{proof}

\subsection{Regionally proximal relation of higher order and Nil-factors}
For a t.d.s. $(X,T)$, Host, Kra and Maass \cite{HKM} introduced the following definition.
If ${\bf n} = (n_1,\ldots, n_d)\in \Z^d$ and $\ep\in \{0,1\}^d$, we
define
$${\bf n}\cdot \ep = \sum_{i=1}^d n_i\ep_i .$$

\begin{de}
Let $(X, T)$ be a t.d.s. and let $d\in \N$. The points $x, y \in X$ are
said to be {\em regionally proximal of order $d$} if for any $\d  >
0$, there exist $x', y'\in X$ and a vector ${\bf n} = (n_1,\ldots ,
n_d)\in\Z^d$ such that $\rho(x, x') < \d, \rho (y, y') <\d$, and $$
\rho (T^{{\bf n}\cdot \ep}x', T^{{\bf n}\cdot \ep}y') < \d\
\text{for any $\ep\in \{0,1\}^d\setminus \{\bf 0\}$},$$
where ${\bf 0}=(0,0,\ldots,0)\in \{0,1\}^d$. The set of regionally proximal pairs of
order $d$ is denoted by $\RP^{[d]}$ (or by $\RP^{[d]}(X,T)$ in case of
ambiguity), and is called {\em the regionally proximal relation of
order $d$}.
\end{de}

Similarly we can define $\RP^{[d]}(X,G)$ for a system $(X,G)$ with $G$ being abelian.
We note that $\RP^{[1]}=\RP$.  It is easy to see that $\RP^{[d]}$ is a closed and invariant
relation. Observe that
\begin{equation*}
    {\bf P}(X)\subset  \ldots \subset \RP^{[d+1]}\subset
    \RP^{[d]}\subset \ldots \subset \RP^{[2]}\subset \RP^{[1]}=\RP.
\end{equation*}

Host, Kra and Maass \cite{HKM} showed that if a system is minimal and
distal then $\RP^{[d]}$ is an equivalence relation, and a very deep result stating that
$(X/\RP^{[d]},T)$ is the maximal $d$-step pro-nilfactor of the system.
Shao and Ye \cite{SY} showed that all these results in fact hold for arbitrarily minimal systems
of $\Z$-actions, and for abelian group actions, $\RP^{[d]}$ is an equivalence relation.
See Glasner, Gutman and Ye \cite{GGY18} for similar results regarding general group actions.

It follows that for any minimal system $(X,T)$, ${\bf RP}^{[\infty]}:=\bigcap_{d=1}^{\infty}{\bf RP}^{[d]}$ is a closed invariant equivalence relation. Let $X_{\infty}$ denote the quotient system of $(X,T)$ by ${\bf RP}^{[\infty]}$.  In \cite{DDMSY} , the notion of systems of order $\infty$ was introduced. A minimal system $(X,T)$ is an {\it $\infty$-step pro-nilsystem} or {\it a system of order $\infty$}, if ${\bf RP}^{[\infty]}=\Delta(X)$. The authors in \cite[Theorem 3.6]{DDMSY} further showed that a minimal system is an $\infty$-step pro-nilsystem if and only if it is an inverse limit of minimal nilsystems.  One can show that $X_{\infty} $ is the maximal factor of $(X,T)$ being order of $\infty$.

\subsection{Saturatedness and topological characteristic factors}
\begin{de}
Let $\phi: X\ra Y$ be a map between two sets. A subset $L$ of $X$ is $\phi$-{\it saturated} if $\phi^{-1}(\phi(L))=L$.
\end{de}
The following lemma is direct by the definition. For the completeness we afford a proof here.
\begin{lem}\label{lem-saturated-basic}
	Let $X,Y,Z$ be sets, and let $\pi: X\rightarrow Y, \phi: X\rightarrow Z, \psi: Z\rightarrow Y$ be
	surjective maps such that $\pi= \psi\circ \phi$
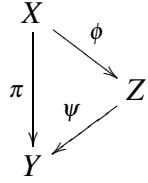
	$$\xymatrix@R=0.5cm{
		X \ar[dd]_{\pi} \ar[dr]^{\phi}             \\
		& Z \ar[dl]_{\psi}         \\
		Y                 }
	$$
	\begin{enumerate}
		\item If $A\subset X$ is $\pi$ saturated, then $A$ is $\phi$ saturated and $\phi(A)$ is $\psi$ saturated.
		\item If $A\subset X$ is $\phi$ saturated and $\phi(A)$ is $\psi$ saturated, then $A$ is $\pi$ saturated.
	\end{enumerate}
\end{lem}
\begin{proof} (1) Since $\pi^{-1}\circ\pi(A)=A$, we get $\phi^{-1}\circ\psi^{-1}\circ\psi \circ\phi(A)=A$.
That is, $\psi^{-1}\circ\psi (\phi(A))=\phi(A).$
	
At the same time, $\phi^{-1}\circ\phi(A)\subset \phi^{-1}\circ\psi^{-1}\circ\psi(\phi(A))=\pi^{-1}\circ\pi(A)=A$. That is, $\phi^{-1}\circ\phi(A)=A$.
	
\medskip
(2) We know that $\psi^{-1}\circ\psi(\phi(A))=\phi(A)$ and $\phi^{-1}\circ\phi(A)=A.$  Thus, $$\pi^{-1}\circ\pi(A)=\phi^{-1}\circ\psi^{-1}\circ\psi(\phi(A))=\phi^{-1}\circ\phi(A)=A.$$
\end{proof}

Let $X, Y$ be compact metric spaces and $\phi: X\ra Y$ a map between them. For $d\in\N$ with $d\geq 2$, we use  $\phi^{(d)}$ to denote the map $$\phi\times\cdots\times \phi: X^d\ra Y^d, (x_1,\cdots,x_d)\mapsto (\phi(x_1),\cdots,\phi(x_d)).$$

\medskip
Let $(X,T)$ be a t.d.s. and $d\geq 2$. Define
$$\sigma_{d}(T)=T^{(d)}=T\times\cdots\times T (d-{\rm times}), ~~ \tau_{d}(T)=T\times T^2\times\cdots\times T^{d},$$
and $\G_{d}(T)=\langle \sigma_{d}(T), \tau_{d}(T)\rangle$ be the group generated by $\sigma_{d}(T)$ and $\tau_{d}(T)$.
Let
\[ N_{d}(X,T)=\overline{\O}(\Delta_{d}(X), \G_{d}(T))=
\overline{\left\{ (T^{p+q}x, T^{p+2q}x,\cdots, T^{p+dq}x):~x\in \Delta_{d}(X), p,q\in \Z\right\}}. \]
When there is no ambiguity we also use $N_d(T)$ or $N_d(X)$ to denote $N_{d}(X,T)$ for short.
 In \cite{G94}, Glasner showed that $(N_{d}(X,T), \G_{d}(T))$ is also a minimal system provided that $(X,T)$ is minimal.

\begin{de}
Let $\pi: (X, T)\ra (Y,T)$ be a factor map between minimal systems and $d\in\N$ with $d\geq 2$. If there exists a dense $G_{\delta}$ subset $\Omega$ of $X$ such that for any $x\in \Omega$, the set
$$L_{x}:=\overline{\O}(x^{(d)}, \tau_{d}(T))=\overline{\left\{ (T^{n}x, T^{2n}x,\cdots, T^{dn}x):~n\in \Z\right\}} $$
is $\pi^{(d)}$-saturated, i.e., $\pi^{(-d)}\left(\pi^{(d)}(L_x)\right)=L_x$, then we say that $(Y,T)$ is a {\it $d$-step topological characteristic factor} of $(X,T)$.
\end{de}

 The following theorems characterizing topological characteristic factors  were proved in \cite{GHSWY}.

\begin{thm}\cite[Theorem A]{GHSWY}\label{5p-mianthm} Let $(X,T)$ be a minimal system, and $\pi:X\rightarrow X_\infty$ be the factor map.
	Then there are minimal systems $X^*$ and $X_\infty^*$ which are almost one to one
	extensions of $X$ and $X_\infty$ respectively, and a commuting diagram below such that  $X_\infty^*$ is a
	$d$-step topological characteristic factor of $X^*$ for all $d\ge 2$,
	\[
	\begin{CD}
		X @<{\sigma^*}<< X^*\\
		@VV{\pi}V      @VV{\pi^*}V\\
		X_\infty @<{\tau^*}<< X_\infty^*
	\end{CD}
	\]
\end{thm}

\begin{thm}\cite[Theorem 4.2]{GHSWY}\label{lab55}
	Let $\pi:(X,T)\rightarrow(Y,T)$ be an extension of minimal systems. If $\pi$ is open and $X_\infty$ is a factor of $Y$, then $Y$ is a $d$-step topological characteristic factor of $X$ for all $d\in \N$.
	$$\xymatrix{
		& X \ar[dl]_{\pi_\infty} \ar[d]^{\pi}             \\
		X_{\infty}  & Y \ar[l]_{\phi}}
	$$
\end{thm}

\begin{lem}\cite[Lemma 3.3]{GHSWY}\label{lem-saturated}
Let $\pi: (X,T)\rightarrow (Y,T)$ be an open extension of minimal systems and $d\in \N$. If $Y$ is
a $d$-step topological characteristic factor of $X$, then $N_{d+1}(X)$ is $\pi^{(d+1)}$-saturated, i.e.
$(\pi^{(d+1)})^{-1}(N_{d+1}(Y))=N_{d+1}(X).$
\end{lem}

We remark that the openness assumptions above are necessary.

\section{On the minimality and total minimality}

In this section we will give a  proof of Theorem A.  First we give the following lemma.
\begin{lem}\label{minimality of Tn}
Let $(X,T)$ be a minimal system with $x\in X$ and $n,d\in\N$.  For any $i,j\in\{0,1,\cdots,n-1\}$, let
$$ [N_{d}(T)]_{j}^{i}=(\id\times T\times\cdots\times T^{d-1})^{i}(T\times T\times\cdots\times T)^{j} N_d(T^n).$$
Then
\begin{enumerate}
\item $N_d(T)=\bigcup_{i=0}^{n-1}\bigcup_{j=0}^{n-1}[N_{d}(T)]_{j}^{i}$;
\item each member of $\{[N_{d}(T)]_{j}^{i}:i,j\in\{0,1,\cdots,n-1\}\}$ is minimal under the action of $\G_d(T^n)$ and any two members are  either identical or disjoint;
\item if $(X, T^n)$ is minimal, then $N_d(T)=\bigcup_{i=0}^{n-1}[N_{d}(T)]^i$ and each $[N_d(T)]^i$ is minimal under the action of $\G_{d}(T^n)$, where $[N_{d}(T)]^i=\bigcup_{j=0}^{n-1}[N_{d}(T)]^i_j$;
\item if for any $l\in \Z$ there exist $y\in X$ and a sequence $(q_i)$ of integers such that
\begin{equation}\label{newview}
T^{nq_i}x\ra T^l y, T^{2nq_i}x\ra T^{2l}y, \cdots, T^{(d+1)nq_i}x\ra T^{(d+1)l}y,
\end{equation}
then $N_{d+1}(T)=N_{d+1}(T^n)$.
\end{enumerate}
\end{lem}
\begin{proof}
It is clear that
\begin{equation*}
\G_d(T)=\bigcup_{i=0}^{n-1}\bigcup_{j=0}^{n-1}
(\id\times T\times\cdots\times T^{d-1})^{i}(T\times T\times\cdots\times T)^{j}\G_{d}(T^{n}).
\end{equation*}
Thus (1) holds and the commutativity of $\G_{d}(T)$ implies that (2) holds. If $(X,T^n)$ is minimal, then
\begin{eqnarray*}
N_d(T)&=&\overline{\G_d(T)x}=\overline{\langle \tau_d(T)\rangle(\overline{\langle \sigma_{d}(T)\rangle x^{(d)}})}
=\overline{\langle \tau_d(T)\rangle(\overline{\langle \sigma_{d}(T^n)\rangle x^{(d)}})} \\
&=& \bigcup_{i=0}^{n-1}(\id\times T\times\cdots\times T^{d-1})^{i} \overline{\langle \tau_d(T^n)\rangle(\overline{\langle \sigma_{d}(T^n)\rangle x^{(d)}})} \\
&=& \bigcup_{i=0}^{n-1}(\id\times T\times\cdots\times T^{d-1})^{i}N_{d}(T^n)=\bigcup_{i=0}^{n-1}[N_{d}(T)]^i.
\end{eqnarray*}
Thus (3) holds.

Finally, we show (4).
First we show that $(X,T^n)$ is minimal. To see this,  let $Y_j$ to be the orbit closure of $T^jx$ under $T^n$ for $0\le j\le n-1$.
For $l=1$, we have $y$ and $(q_i)$ such that (\ref{newview}) holds. We assume that $y\in Y_{j_0}$. Thus, we have
$Y_0=Y_{j_0+1}=Y_{j_0+2}.$ This clearly implies that $Y_0=Y_j$ for all $1\le j\le n-1$ and
hence $(X,T^n)$ is minimal.

Now we show $N_{d+1}(T)=N_{d+1}(T^n)$. For any $l\in \Z$, we have $y$ and $(q_i)$ such that (\ref{newview}) holds.
For any $k\in\Z$ there is a sequence $(p_i)$ of integers such that
\begin{equation}\label{total-minimal}
T^{np_i}x\lra T^ky,
\end{equation} since $(X,T^n)$ is minimal. Thus, combining (\ref{newview}) and (\ref{total-minimal}) we get that
there are subsequences $(p_i')$ and
$(q_i')$ of $(p_i)$ and $(q_i)$ such that
$$T^{p_i'+nq_i'}x\ra T^{k+l} y, T^{p_i'+2nq_i'}x\ra T^{k+2l}y, \cdots, T^{p_i'+(d+1)nq_i'}x\ra T^{k+(d+1)l}y.$$

This implies that $\sigma^k\tau^l y^{(d+1)}\subset N_{d+1}(T^n)$ and in turn implies
$\sigma^k\tau^l x^{(d+1)}\subset N_{d+1}(T^n)$ by the fact that $(X,T^n)$ is minimal.
Since $l$ and $k$ are arbitrary, we get that $N_{d+1}(T)\subset N_{d+1}(T^n)$.
\end{proof}

Next we give two equivalent characterizations for the equality between $N_{d}(T)$ and $N_{d}(T^n)$. The equivalence between (1) and (3) in the following theorem was given in \cite[Theorem 6.2]{GHSWY}.  However, the equivalence with (2) is new, which is the main ingredient for our proof of Theorem A.
\begin{thm}\label{equivalence} Let $(X, T)$ be a minimal system, and $d,n\in\N$. Then the following statements are equivalent.
\begin{enumerate}
\item $N_{d+1}(T)=N_{d+1}(T^n)$.

\item There is a dense $G_\delta$ subset $X_0$ of $X$ such that for any  $x\in X_0$,  there exists $y\in X$ such that for any $l\in\Z$ there is a sequence $(q_i)$ of integers satisfying
$$T^{nq_i}x\lra T^ly, T^{2nq_i}x\lra T^{2l}y, \ldots, T^{(d+1)nq_i}x\lra T^{(d+1)l}y.$$

\item There is a dense $G_\delta$ subset $X_0$ of $X$ such that for any $l\in \Z$ and $x\in X_0$, there is a sequence $(q_i)$ of integers
such that
$$T^{nq_i}x\lra T^lx, T^{2nq_i}x\lra T^{2l}x, \ldots, T^{dnq_i}x\lra T^{dl}x.$$
\end{enumerate}
\end{thm}
\begin{proof} It remains to show the equivalence between (1) and (2). It is clear that (2) implies (1), by Lemma \ref{minimality of Tn} (4). Next we are going to show that (1) implies (2).

\medskip
First we show the case $d=1$.
By the assumption $N_2(T)=N_2(T^n)$ we know that for a given $x\in X$ and $l\in\Z$, there are sequences $(p_i)$ and $(q_i)$
of integers such that
\begin{equation}\label{2terms}
T^{np_i+nq_i}x\lra T^lx, T^{np_i+2nq_i}x\lra T^{2l}x.
\end{equation}
Given $\ep>0$ and $l\in\Z$, there are $x'\in X$ and $q\in\Z$ such that
$$\rho(T^{nq}(x'),T^lx)<\ep,\ \ \rho(T^{2nq}(x'),T^{2l}x)<\ep,$$
with $x'=T^{np}x$ for some $p\in\Z$ by (\ref{2terms}).

Let $k\in\N$ and
$$A_k=\left\{z\in X: \exists z'\in X\  s.t.\ \forall l\in\Z\ \exists q\in\Z \text{ with }\ \rho(T^{nq}(z),T^lz'), \rho(T^{2nq}(z),T^{2l}z')<\frac{1}{k}\right\}.$$
It is clear that $A_k$ is non-empty and open. To show it is dense we need to use the minimality of $(X,T)$. To see this let
$U$ be a non-empty open subset of $X$ and assume that $X=\bigcup_{i=1}^{N}T^iU$ for some $N\in\N$. Choose $j>0$ such that $\rho(x_1,x_2)<\frac{1}{j}$ implies that $\rho(T^ix_1,T^ix_2)<\frac{1}{k}$ for $1\le i\le N$.  Given $l\in\Z$, choose $q\in\Z$ satisfying $\rho(T^{nq}(x'),T^lx)<\frac{1}{j}, \rho(T^{2nq}(x'),T^{2l}x)<\frac{1}{j}$ and assume that $T^tx'\in U$ for some $1\le t\le N$. We have
$$\rho(T^{nq}(T^tx'),T^l(T^tx))<\frac{1}{k},\ \ \rho(T^{2nq}(T^tx'),T^{2l}(T^tx))<\frac{1}{k}$$
and hence $T^tx'\in U\cap A_k$.

Let $X_0= \bigcap_{k=1}^\infty A_k$ and $w\in X_0$, then  there is $w_k\in X$ and sequences $(q_k^{(l)})$ such that for any $l\in \Z$, we have
$$\rho(T^{nq^{(l)}_k}(w),T^lw_k)<\frac{1}{k}, \rho(T^{2nq^{(l)}_k}(w),T^{2l}w_k)<\frac{1}{k}.$$
We may assume that $\lim w_k=w'$ we have $T^{nq^{(l)}_k}(w)\lra T^lw'$ and $T^{2nq^{(l)}_k}(w)\lra T^{2l}w'$, for any $l\in\Z$.

\bigskip
Net we turn to the general case. By the assumption $N_{d+1}(T)=N_{d+1}(T^n)$ we know that for a given $x\in X$ and  $l\in\Z$, there are sequences $\{p_i\}$ and $\{q_i\}$ such that
\begin{equation}\label{3terms}
	T^{np_i+nq_i}x\lra T^lx, T^{np_i+2nq_i}x\lra T^{2l}x,\ldots,T^{np_i+(d+1)nq_i}x\lra T^{(d+1)l}x.
\end{equation}
Given $\ep>0$ and $l\in\Z$,  there is $x'\in X$ and $q\in\Z$ such that
$$\rho(T^{nq}(x'),T^lx)<\ep,\ \ \rho(T^{2nq}(x'),T^{2l}x)<\ep,\ldots,\rho(T^{(d+1)nq}(x'),T^{(d+1)l}x)<\ep$$
with $x'=T^{np}x$ for some $p\in\Z$ by (\ref{3terms}).

For $k\in\N$ let
$$A_k=\left\{z\in X: \exists z'\in X\  s.t.\ \forall l\in\Z\ \exists q\in\Z\ \text{with}\ \rho(T^{inq}(z),T^{il}z')<\frac{1}{k}\  \text{for} \ i=1,\ldots,d+1\right\}.$$

Let $U$ be a non-empty open subset of $X$ and assume that $X=\bigcup_{i=1}^{N}T^iU$ for some $N\in\N$. Choose $j>0$ such that $\rho(x_1,x_2)<\frac{1}{j}$ implies that $\rho(T^ix_1,T^ix_2)<\frac{1}{k}$ for $1\le i\le N$.  Given $l\in\Z$, pick $q\in\Z$ with $\rho(T^{inq}(y),T^{il}x)<\frac{1}{j}$ for any $i=1,2,\ldots,d+1$, and assume that $T^tx'\in U$ for some $1\le t\le N$. We have
$$\rho(T^{inq}(T^tx'),T^{il}(T^tx))<\frac{1}{k}$$ for any $i=1,\ldots,d+1$ and hence $T^tx'\in U\cap A_k$.

Let $X_0= \bigcap_{k=1}^\infty A_k$ and $w\in X_0$, then for any $i\in\{1,2,\ldots,d+1\}$ there is $w_k\in X$ and  sequences $(q_k^{(l)})$ and  such that for any $l\in\Z$, we have
$$\rho(T^{inq_k^{(l)}}(w),T^{il}w_k)<\frac{1}{k}.$$
We may assume that $\lim w_k=w'$  then for any $l\in\Z$, we have $$T^{nq_k^{(l)}}(w)\lra T^lw',T^{2nq_k^{(l)}}(w)\lra T^{2l}w',\ldots, T^{(d+1)nq_k^{(l)}}(w)\lra T^{(d+1)l}w'.$$
\end{proof}

We remark that in fact the condition (4) in Lemma \ref{minimality of Tn}
is also equivalent to any condition in the above theorem.

\begin{lem}\label{splits-minimal}
Let $(X,T)$ be a minimal system and $n\ge 2$. Assume that $N_d(T)=\bigcup_{i,j=0}^{n-1}[N_{d}(T)]_{j}^{i}$,
where $[N_d(T)]_j^i=\tau^i\sigma^jN_d(T^n)$.
If $(x,y)\in \RP(N_d)$ then there are $1\le i,j\le n-1$ such that $x,y\in [N_d(T)]^i_j$.

\end{lem}
\begin{proof} Note that each element of $W=\{[N_d(T)]_j^i: 0\le i,j\le n-1\}$ is minimal under the $\G(T^n)$-action.
Thus, any two members are either identical or disjoint. Let $W=\{A_1,\ldots,A_k\}$ with $A_i\cap A_j=\emptyset$ if $i\not=j$.
It is clear for each $g\in \G_d(T^n)$, and each $1\le i\le k$, $g(A_i)\in W$. Then the lemma follows by the definition of $\RP(N_d)$.

Precisely, let $\ep_0=\min\{\rho(A_i,A_j):1\le i,j\le k\}$. For $0<\ep<\ep_0$, there are $x',y'\in N_d(T)$ and $g\in \G(T^n)$
such that $\rho(x,x')<\ep$, $\rho(y,y')<\ep$ and $\rho(gx',gy')<\ep$. It is clear that $x,x'\in A_i$ and $y,y'\in A_j$ for some $1\le i,j\le k$.
Since $gx',gy'$ is in the same element of $W$, it follows that $i=j$.
\end{proof}


\begin{thm}\cite[Theorem B]{GHSWY}\label{5-B}
Let $(X,T)$ be a minimal system and $d\in\N$. Then
the maximal equicontinuous factor of $(N_d(X,T), \langle\sigma_d, \tau_d\rangle)$
is $(N_d(X_1,T), \langle\sigma_d, \tau_d\rangle)$, where as above $X_1$ is the maximal equicontinuous factor of $(X,T)$.
\end{thm}

We now give another approach of Theorem A based on Theorem \ref{5-B} and Theorem \ref{equivalence}.
\begin{proof}[Proof of Theorem A]
It suffices to show that the minimality of $(X, T^{n})$ implies the equality  $N_d(T^n)=N_d(T)$ for any $d\in\N$.

 By Lemma \ref{minimality of Tn}, we have $N_d(T)=[N_d(T)]^{0}\cup [N_d(T)]^{1}\cup \cdots\cup [N_d(T)]^{n-1}$ and each $[N_d(T)]^{i}$ is $\G_d(T^n)$-minimal. Thus any pair of $[N_d(T)]^i$ and $[N_d(T)]^j$ are either identical or disjoint.

\medskip
Next we are going to show that $[N_d(T)]^{i}=[N_d(T)]^{1}$ for each $i\in\{0,1,\cdots,n-1\}$ by induction on $d$, which will implies that $N_d(T)=N_d(T^n)$, by noting that $N_{d}(T^n)=[N_{d}(T)]^{0}$.

\medskip
For $d=1,2$, it is clear that $N_1(T)=N_1(T^n)=X$ and $N_2(T)=N_2(T^n)=X\times X$, by the minimality of $(X,T^n)$. Now we assume that $N_d(T)=N_d(T^n)$ holds for some $d\geq 2$. We are going to show $N_{d+1}(T)=N_{d+1}(T^n)$.

\medskip
For $j\in\{0,1,\cdots, n-1\}$, by taking $l=-j$ in Theorem \ref{equivalence} (2), there is a dense $G_{\delta}$ subset $X_0$ of $X$ such that for each $x\in X_0$, there is some $u=u(x)\in X$ and a sequence $(q_i)$ of integers with  $q_i\equiv j\mod n$ satisfying
\begin{equation}\label{eq3.1}
T^{q_i}x\ra u, T^{2q_i}x\ra u,\ldots,T^{dq_i}x\ra u,
\end{equation}
which implies that $(x,u)\in\RP(X,T)$ and $(x,u,\cdots,u)\in [N_{d+1}(T)]^{j}$. Thus $$(x^{(d+1)}, (x,u,\cdots,u))\in\RP (N_{d+1}(T))$$
according to Theorem \ref{5-B}. By noting that $x^{(d+1)}\in [N_{d+1}(T)]^{0}$ and Lemma \ref{splits-minimal}, we have $[N_{d+1}(T)]^{0}=[N_{d+1}(T)]^{j}$ whence  $N_{d+1}(T)=N_{d+1}(T^n)$. This completes the proof.
\end{proof}

For a minimal distal system $(X,T)$, we have an approach of Theorem A (without using Theorem \ref{5-B})
by using the following theorem 
which is a direct consequence of Theorem \ref{5p-mianthm}.
\begin{thm}\cite[Theorem 4.3]{GHSWY}\label{main-distal}
Let $(X,T)$ be a minimal system which is an open extension of its maximal distal factor,
and $d\in \N$. Then $X_d$ is a $(d+1)$-step topological characteristic factor of $X$.
\end{thm}

\begin{proof}[Proof of Theorem A for minimal distal systems] We use induction on $d$. It is trivial for $d=1,2$.
Now assume that $N_d(T)=N_d(T^n)$ for some $d\geq 2$.  Then, by applying Theorem \ref{equivalence} (2), there is a dense $G_\delta$ subset $X_0$ such that for any $x\in X_0$ and each $l\in\mathbb{Z}$,
there is some $u=u(x)\in X$ and a sequence $(q_i)$ of integers with $q_i\equiv -l\mod n$ satisfying
\begin{equation}\label{eq3.3}
T^{q_i}x\lra u, T^{2q_i}x\lra u,\ldots,T^{dq_i}x\lra u,
\end{equation}
which implies that $(x,u)\in \RP^{[d-1]}(X,T)=\RP^{[d-1]}(X,T^n)$. Since $(X,T^n)$ is minimal, the maximal $d-1$-step pronilfactor
of $(X,T^n)$ is the same as the one of $(X,T)$.
WLOG we assume that $X_0$ is also the set of saturation points both for $T$ and $T^k$, which is also a dense $G_{\delta}$ set.

Since $(X,T)$ is distal, $L_x$ is minimal and $(X,T^k)$ is minimal, we know that $\pi_{d-1}^{-1}\pi_{d-1}(x)\subset X_0$,
where $\pi_{d-1}:X\lra X_{d-1}$ is the factor map.
By Theorem \ref{main-distal} (applying to $(X,T^n)$) there is a sequence $(p_i)$ of integers with $p_i\equiv 0\mod n$ for each $i$ such that
\begin{equation}\label{eq3.4}
T^{p_i}u\lra x, T^{2p_i}u\lra x,\ldots, T^{dp_i}u\lra x,
\end{equation}
Combining (\ref{eq3.3}) and (\ref{eq3.4}),  there is a sequence $(r_i)$ of integers satisfying
\begin{equation*}
T^{nr_i}x\lra T^{l}x, T^{2nr_i}x\lra T^{2l}x,\ldots, T^{dnr_i}x\lra T^{dl}x,
\end{equation*}
which implies that $N_{d+1}(T)=N_{d+1}(T^n)$ according to Theorem \ref{equivalence}. Thus we complete the proof.
\end{proof}

\section{Maximal distal factors}

In this section, we  will show that the maximal distal factor of $N_d(X)$ is $N_{d}(X_{dis})$.

\subsection{Closed proximal relations} In this subsection, we characterize the maximal distal factor of $N_{d}(X)$ under the condition that the proximal relation of $X$ is closed.

\begin{lem}\label{liftproximal}
Let $(X,T)$ be a t.d.s. and $d\in \N$.  Then
\begin{enumerate}
\item for any $n\in\N$, ${\bf P}(X, T)={\bf P}(X, T^n)$ ;
\item If ${\bf P}(X)$ is closed, then
\begin{equation}\label{eq4.2}
{\bf P}(N_d(X))=\{((x_i)_{1}^{d},(y_i)_{1}^{d})\in N_{d}(X)\times N_d(X): (x_i,y_i)\in {\bf P}(X,T), i=1,\cdots,d\}.
\end{equation}
\end{enumerate}
\end{lem}
\begin{proof}
(1) It is clear that ${\bf P}(X, T^n)\subset{\bf P}(X, T)$. By the proof of Lemma \ref{finite index subgroup}, we have
${\bf P}(X, T)\subset{\bf P}(X, T^n)$. Thus (1) holds.

\medskip
(2) It is also clear that ${\bf P}(N_d(X))\subset {\bf RHS}$ of (\ref{eq4.2}), where ${\bf RHS}$ stands for the right hand side.
Next we show ${\bf P}(N_d(X))\supset {\bf RHS}$ of (\ref{eq4.2}) by induction on $d$. It is obvious for $d=1$. Now assume it holds for $d$.

Suppose that  $((x_i)_{1}^{d+1},(y_i)_{1}^{d+1})\in N_{d+1}(X)\times N_{d+1}(X)$ satisfying for any $i=1,\cdots,d+1, (x_i,y_i)\in {\bf P}(X,T)$.
By induction assumption, we have $((x_i)_{1}^{d},(y_i)_{1}^{d})\in{\bf P}(N_d(X))$. Then there exist sequences $(p_i)$ and $(q_i)$ of integers such that
$$T^{p_i}\times T^{p_i+q_i}\times T^{p_i+2q_i}\times\cdots\times T^{p_i+(d-1)q_i}((x_i)_{1}^{d},(y_i)_{1}^{d})\lra \text{ a point in }\Delta(N_d(X)).$$
By passing to some subsequences, we may assume that $T^{p_i+dq_i}x_{d+1}\ra z_1$ and $T^{p_i+dq_i}y_{d+1}\ra z_2$ for some $z_1,z_2\in X$. Since ${\bf P}(X)$ is closed, we have $(z_1,z_2)\in {\bf P}(X)$. Thus there is a sequence $(r_i)$ of integers such that $T^{r_i}(z_1,z_2)$ tends to a point in $\Delta(X)$. For any $k\in \N$, by uniform continuity, there exists some positive integers $m_k$ and $n_{k}$ such that for any $n\geq n_k$,
$$\rho\left( T^{r_{m_k}+p_{n}}\times\cdots\times T^{r_{m_k}+p_{n}+dq_{n}}(x_i)_{1}^{d+1}, T^{r_{m_k}+p_{n}}\times\cdots\times T^{r_{m_k}+p_{n}+dq_{n}}(y_i)_{1}^{d+1}\right)<\frac{1}{k}.$$
We may assume that $(m_k)$ and $(n_k)$ are increasing. Then
$$\rho\left( T^{r_{m_k}+p_{n_k}}\times\cdots\times T^{r_{m_k}+p_{n}+dq_{n_k}}(x_i)_{1}^{d+1}, T^{r_{m_k}+p_{n_k}}\times\cdots\times T^{r_{m_k}+p_{n_k}+dq_{n_k}}(y_i)_{1}^{d+1}\right)\ra 0.$$
So $((x_i)_{1}^{d+1},(y_i)_{1}^{d+1})\in{\bf P}(N_{d+1}(X))$. Thus (2) holds.
 \end{proof}

\begin{thm}\label{cl prox}
	Let $(X,T)$ be a minimal t.d.s.
	If the proximal relation of $(X,T)$ is closed and $d\in \N$, then the maximal distal factor of $(N_d(X,T),\G_d)$ is $(N_d(X_{dis},T),\G_d)$, where $X_{dis}$ is the maximal distal factor of $(X,T)$.
\end{thm}	
\begin{proof}
By Lemma \ref{liftproximal}, we have
\begin{equation}\label{eq5.3}
{\bf P}(N_d(X))=\{((x_i)_{1}^{d},(y_i)_{1}^{d})\in N_{d}(X)\times N_d(X): (x_i,y_i)\in {\bf P}(X,T), i=1,\cdots,d\}.
\end{equation}
Thus ${\bf P}(N_d(X))$ is closed. Recall that if a proximal relation is closed, then it is an equivalence relation (see \cite[Chapter 6, Corollary 11]{Aus}). It follows that $X_{dis}=X/{\bf P}(X)$ and $(N_d(X))_{dis}=N_d(X)/{\bf P}(N_d(X))$. By (\ref{eq5.3}), we conclude that $(N_d(X))_{dis}=N_d(X_{dis})$.

\end{proof}

We end the subsection with the following remark.

\begin{rem}We remark that even in the assumption $(X,T)$ is minimal and $\pi:X\lra X_{eq}$ is almost 1-1, it is not true that
$(x_1,x_2,x_3)\in N_3$ if $\pi(x_1)=\pi(x_2)=\pi(x_3)$, as $N_3$ can be seen as a subset of $\mathbf{Q}^{[2]}$, see \cite{TY}.
\end{rem}
	
\subsection{Open extensions}
 We start with some properties of topological characteristic factors.

\bigskip
Let $Z$ be a compact metric space. Let $2^{Z}$ be the space of nonempty closed subsets of $Z$, which is also a compact metric space endowed with Hausdorff metric. For a sequence $(A_n)$ in $2^{Z}$, define
$\liminf A_n=\{z\in Z: \exists z_n\in A_n~~s.t.~~ z_n\ra z\}$ and  $$ \limsup A_n=\{z\in Z: \exists \text{ subsequence } (n_i) \text{ and }z_{n_i}\in A_{n_i}~~s.t.~~ z_{n_i}\ra z\}.$$
Let $Y$ be a metric space. A map $f: Y\ra 2^{Z}$ is {\it lower semi-continuous} (resp. {\it upper semi-continuous}) at $y\in Y$ if for any sequence $(y_n)$ with $y_n\ra y$, $\liminf f(y_n)\supset f(y)$ (resp. $\limsup f(y_n)\subset f(y)$ ). Note that $f$ is lower semi-continuous  at $y$ if for any open set $V$ with $f(y)\cap V\neq\emptyset$, $\{y'\in Y: f(y')\cap V\neq\emptyset\}$  is a neighborhood of $y$, and
$f$ is upper semi-continuous at $y$ if for any open set $V$ of $Z$ containing $f(y)$,   $\{y'\in Y: f(y')\subset V\}$ is a neighborhood of $y$.
\begin{lem}\label{thm:genericity}
Let $(X,T)$ be a minimal system and $d\ge 2$, then for a dense $G_{\delta}$ subset
$X_{0}\subset X$ such that for $x\in X_0$ one has $N_d[x]:=\{(x_1,\cdots,x_d)\in N_d(X):x_1=x\}=\{x\}\times L_x$, where $L_x$
is the orbit closure of $x^{(d-1)}$ under $T\times T^2\times \ldots \times T^{d-1}$ action.
\end{lem}
\begin{proof}
Consider $\Phi:X\to 2^{X^d}$ given by $x\mapsto \{x\}\times L_x$.
It is easy to check that this map is lower-semi-continuous (see \cite[Lemma 6.1, Claim 1]{GHSWY}).  It is well known that the set of continuity points of $\Phi$ is a dense $G_{\delta}$ subset
$X_{0}\subset X$. We now show $\cup_{x\in X}\{x\}\times L_x$ is dense in $N_d(X)$.
To see this let $(y_1,\ldots,y_d)\in N_d$. Then there are subsequences $\{p_i\}$ and $\{q_i\}$ of $\Z$ and $x\in X$
such that $T^{p_i}x\lra y_1, T^{p_i+q_i}x\lra y_2, \ldots, T^{p_i+(d-1)q_i}x\lra y_d$.
Let $z_i=T^{p_i}x$, we have $z_i\lra y_1, T^{q_i}z_i\lra y_2, \ldots, T^{(d-1)q_i}z_i\lra y_d$.
It is clear that $(z_i, T^{q_i}z_i, \ldots, T^{(d-1)q_i}z_i)\in \{z_i\}\times L_{z_i}$.

Thus, it follows that at each point $x\in X_{0}$ we must have $N_d[x]=\{x\}\times L_x$.
Indeed let $x_{0}\in X_{0}$ and assume $N_d[x]\not=\{x\}\times L_x$, i.e. $\{x\}\times L_x\varsubsetneq N_d[x]$.
Let $U$ be an open set in $N_d(X)$ so that
$$\{x_0\}\times L_{x_0}\subset N_d(X)\cap U\neq N_d[x].$$
As $\Phi$ is  continuous at $x_{0}$
the set $\{x\in X|\: \{x\}\times L_x\subset U\}$
is a neighborhood of $x_{0}$ and it follows $\cup_{x\in X}\{x\}\times L_x$ is not dense in $N_d(X)$.
\end{proof}

To sum up we have
\begin{thm}\label{label53}Let $\pi: (X,T)\rightarrow (Y,T)$ be an extension of minimal systems and $d\in \N$.
\begin{enumerate}
\item If $Y$ is a $d$-step topological characteristic factor of $X$ and $\pi$ is open,
then $N_{d+1}(X)$ is $\pi^{(d+1)}$-saturated, i.e.
$(\pi^{(d+1)})^{-1}(N_{d+1}(Y))=N_{d+1}(X).$

\item If $N_{d+1}(X)$ is $\pi^{(d+1)}$-saturated, then $Y$ is a $d$-step topological characteristic factor of $X$.
\end{enumerate}
\end{thm}
\begin{proof} (1) follows from Lemma \ref{lem-saturated}.

To show (2) we use Lemma \ref{thm:genericity}.
Indeed, by Lemma \ref{thm:genericity} we know that there is a dense $G_\delta$ subset
$X_{0}\subset X$ such that for $x\in X_0$ one has $N_d[x]=\{x\}\times L_x$. Now let $(x_1,\ldots,x_{d-1})\in L_x$,
then we have $(x,x_1,\ldots,x_{d-1})\in N_d(X)$. For any $(y_1,\ldots,y_{d-1})\in X^{(d-1)}$
with $\pi(x_i)=y_i$, $1\le i\le d-1$, we have $(x, y_1,\ldots,y_{d-1})\in N_d(X)$.
Thus, $(x, y_1,\ldots,y_{d-1})\in N_d[x]=\{x\}\times L_x$, i.e. $(y_1,\ldots,y_{d-1})\in L_x$.
\end{proof}

\subsection{Proof of Theorem B}

First we need a simple lemma.

\begin{lem}\label{distrel}
Let $\pi:X\lra Y$ be an extension between minimal systems.  If $R_{\pi}\subset S_{dis}(X)$, then $X_{dis}=Y_{dis}$. Particularly, this holds for proximal extensions.
\end{lem}
\begin{proof}
It is clear that $Y_{dis}$ is a distal factor of $X$. By the maximality of $X_{dis}$, $Y_{dis}$ is a factor of $X_{dis}$.  If $R_{\pi}\subset S_{dis}(X)$, then $X_{dis}=X/S_{dis}(X)$ is a factor of $Y=X/R_{\pi}$.  Since $X_{dis}$ is distal, we have that $X_{dis}$ is also a factor of $Y_{dis}$. Thus $X_{dis}=Y_{dis}$.

When $\pi$ is proximal we have $R_{\pi}\subset {\bf P}(X)\subset S_{dis}(X)$. Thus the conclusion holds for proximal extensions.
\end{proof}

The following lemma is direct from the definition.
\begin{lem}\label{distfactor}
Let $\pi:X\lra Y$, $Z$ is a factor of $Y$ and distal. If the maximal distal factor of $X$ is $Z$,
then $Z$ is also the maximal distal factor of $Y$.
\end{lem}

\begin{proof}[Proof of Theorem B]
Let $\pi: X\lra X_{dis}$ be the factor map. Then there are minimal systems $X^{*}$ and $Y$ such that $\phi: X^{*}\lra X$ and $\psi: Y\lra X_{dis}$ are almost 1-1 extension, $\pi^{*}: X^{*}\lra Y$ is open and the following commutative diagram holds:
\begin{equation}\label{diagram1}
\begin{CD}
X @<{\phi}<< X^*\\
@VV{\pi}V      @VV{\pi^*}V\\
X_{dis} @<{\psi}<< Y
\end{CD}
\end{equation}
By Lemma \ref{distrel}, we have $X_{dis}=(X^{*})_{dis}$. Let $\tau= \pi\circ\phi: X^{*}\lra X_{dis}$.
Then (\ref{diagram1}) naturally induces the following diagram.
$$\xymatrix@R=1.6cm{
		N_{d}(X)  \ar[d]_{\pi^{(d)}}    &  N_{d}(X^*) \ar[d]^{(\pi^*)^{(d)}}\ar[l]_{\phi^{(d)}} \ar[dl]_{\tau^{(d)}}       \\
		 N_{d}(X_{dis})  &   N_{d}(Y) \ar[l]_{\psi^{(d)}}    \\
		               }
	$$

\medskip
\noindent {\bf Claim.} $R_{\tau^{(d)}}\subset S_{dis}(N_{d}(X^{*}))$.

\medskip

If the Claim holds, then by Lemma \ref{distrel} we have $(N_{d}(X^{*}))_{dis}=N_{d}(X_{dis})$. It follows from Lemma \ref{distfactor}   that the maximal distal factor of $N_{d}(X)$ is $N_{d}(X_{dis})$.

\medskip
Now we are going to prove the claim. Fix $\left((x_i)_{1}^{d}, (x'_i)_{1}^{d}\right)\in R_{\tau^{(d)}}$.
Then $(x_i,x'_i)\in S_{dis}(X^{*})$ for each $1\leq i\leq d$.  Let $y_i=\pi^{*}(x_i)$ and $y'_i=\pi^{*}(x'_i)$
for each $1\leq i\leq d$. By Lemma \ref{liftofdis}, we have $(y_i,y'_i)\in S_{dis}(Y)$ for
each $1\leq i\leq d$.  Since $\psi$ is almost 1-1 and $X_{dis}$ is distal, we have $R_{\psi}={\bf P}(Y)=S_{dis}(Y)$.
By Lemma \ref{liftproximal}, we have $R_{\psi^{(d)}}=S_{dis}(N_{d}(Y))={\bf P}(N_{d}(Y))$. Thus there is a
sequence $(g_n)\in \G_{d}$ such that $g_{n}\left(\left((y_i)_{1}^{d}, (y'_i)_{1}^{d}\right)\right)\lra \Delta(N_d(Y))$. Then $\overline{\O}\left(\left((x_i)_{1}^{d}, (x'_i)_{1}^{d}\right)\right)\cap R_{(\pi^{*})^{(d)}}\neq\emptyset$.

Since $X_\infty$ is a factor of $X_{dis}$, it follows from Theorems \ref{lab55} and \ref{label53}
that $N_d(X^{*})$ is $(\pi^{*})^{(d)}$-saturated. Let $\left((z_i)_{1}^{d}, (z'_i)_{1}^{d}\right)\in \overline{\O}\left(\left((x_i)_{1}^{d}, (x'_i)_{1}^{d}\right)\right)\cap R_{(\pi^{*})^{(d)}}$. It follows from the saturatedness of $(\pi^{*})^{(d)}$ that
\[(z_1,\cdots ,z_i,z'_{i+1},\cdots,z'_{d})\in N_{d}(X^{*}),   \]
for each $0\leq i\leq d$.

First consider the points $((z_1,\ldots,z_d),(z'_1,z_2,\ldots,z_d))\in (N_d(X^*))^2$. Since $\pi^*(z_1)=\pi^*(z'_1)$, we have $\tau(z_1)=\tau(z'_1)$ and hence $(z_1,z'_1)\in S_{dis}(X^*)$. It follows from Theorem \ref{JG} that
	$$\overline{\O}((z_1,z'_1),T)\cap \overline{\E(\overline{{\bf P}(X^*)})}\ne\emptyset$$
	which implies
	$$\overline{\O}(((z_1,\ldots,z_d),(z'_1,z_2,\ldots,z_d)),\G_d)\cap \overline{\E(\overline{{\bf P}(N_d(X^*))})}\ne\emptyset.$$
Thus
$$((z_1,\ldots,z_d),(z'_1,z_2,\ldots,z_d))\in S_{dis}(N_d(X^*)).$$
Next consider $((z'_1,z_2,\ldots,z_d),(z'_1,z'_2,z_3\ldots,z_d))\in (N_d(X^*))^2$ in the same way.
Inductively, for any $i=2,3,\ldots,d$, we have
	$$((z'_1,\ldots,z'_{i-1},z_i,\ldots,z_d),(z'_1,\ldots,z'_{i-1},z'_i,z_{i+1},\ldots,z_d))\in S_{dis}(N_d(X^*)).$$
Since $S_{dis}(N_d(X^*))$ is an invariant closed equivalence relation, it follows that
	$$((z_1,\ldots,z_d),(z'_1,\ldots,z'_d))\in S_{dis}(N_d(X^*)).$$
Thus $$\overline{\O}( ((z_1,\ldots,z_d),(z'_1,\ldots,z'_d)) )\cap \overline{\E(\overline{{\bf P}(N_d(X^*))})}\neq\emptyset.$$
 Recall that $\left((z_i)_{1}^{d}, (z'_i)_{1}^{d}\right)\in \overline{\O}\left(\left((x_i)_{1}^{d}, (x'_i)_{1}^{d}\right)\right)$. So we have $$\overline{\O}\left(\left((x_i)_{1}^{d}, (x'_i)_{1}^{d}\right)\right)\cap \overline{\E(\overline{{\bf P}(N_d(X^*))})}\neq\emptyset.$$ Hence $\left((x_i)_{1}^{d}, (x'_i)_{1}^{d}\right)\in S_{dis}(N_{d}(X^*))$. This proves the claim.
\end{proof}

\section{Extensions and the structure theorems}

In this section, we will show that minimal t.d.s. $(X,T)$ and $(N_d(X),\G_d)$ have the same structure theorems.

\subsection{Structure theorem for minimal systems}

 We say that a minimal system $(X,G)$ is a {\it strictly } {\bf PI} {\it system} if there is a countable ordinal $\eta$ and a family of minimal systems $(X_\lambda,G)$ $(0\leq\lambda\leq\eta)$ such that
    \begin{enumerate}
    \item $X_0$ is a singleton;
    \item for every successor ordinal $\lambda<\eta$, there exists an extension $\phi_\lambda: (X_{\lambda+1},G)\rightarrow (X_\lambda,G)$ which is either proximal or equicontinuous;
    \item for a limit ordinal $\lambda<\eta$ the system $(X_\lambda,G)$ is the inverse limit of the systems
    $(X_\mu,G)_{\mu<\lambda}$;
    \item $X_\eta=X$.
    \end{enumerate}

    A minimal system $(X,G)$ is a  {\bf PI} {\it system} if there exists a strictly {\bf PI} system $(\tilde{X},G)$ and a proximal extension $\pi: (\tilde{X},G)\ra (X,G)$.  If we replace the proximal extensions by almost $1-1$ extensions in the definition of strictly {\bf PI} system, the resulting system is called an {\bf HPI} system. If we replace the proximal extensions by trivial extensions, the resulting system is called an {\bf I} system.

\begin{thm}[Structure Theorem of Minimal Distal Systems]
Every minimal distal system is an {\bf I} system.
\end{thm}

An extension $\phi: (X,G)\ra (Y,G)$ is called a {\it weakly mixing extension} if $R_{\phi}$ under the diagonal action of $G$ is topologically transitive.

An extension $\phi: (X,G)\ra (Y,G)$ is called a {\bf RIC} (relatively incontractible) extension if $\pi$ is an open map and  for any $n\geq 1$ the minimal points are dense in $R_{\phi}^{n}=\{(x_1,\cdots,x_n)\in X^{n}: \phi(x_1)=\cdots=\phi(x_n)\}$. The notion of {\bf RIC} extension was introduced in \cite{EGS75}. Here we use an equivalent definition avoiding extra notions (see \cite[Theorem A.2]{AGHSY10}).

\begin{thm} [Structure Theorem of Minimal Systems] \label{STMS}
Let $(X,T)$ be a minimal system. Then there exist a countable ordinal $\eta$ and a canonically defined commutative diagram of minimal systems(a PI tower):
    $$\xymatrix{
    X=X_0 \ar[dr]^{\sigma_1}\ar[d]_{\pi_0}&  & X_1 \ar[d]_{\pi_1}   \ar[ll]_{\phi_1}  & \ldots X_v \ar[dr]^{\sigma_{v+1}}\ar[d]_{\pi_v} \ar[l] & &X_{v+1} \ldots \ar[ll]_{\phi_{v+1}}\ar[d]_{\pi_{v+1}} & X_\eta \ar[d]_{\pi_\eta} \ar[l] \\
    \{pt\}=Y_0 & Z_1 \ar[l]^{\rho_1} & Y_1   \ar[l]^{\psi_1} & \ldots Y_v \ar[l]& Z_{v+1} \ar[l]^{\rho_{v+1}} &Y_{v+1} \ldots \ar[l]^{\psi_{v+1}} &Y_\eta \ar[l]
    }$$
	where for $v \leq\eta$, $\pi_v$ is {\bf RIC}, $\phi_v$ and $\psi_v$ are proximal, $\rho_v$ are equicontinuous. $\pi_\eta $ is {\bf RIC} and weakly mixing. For a limit ordinal $v,X_v,Y_v,\pi_v,$ ect. are the inverse limits of $X_\lambda,Y_\lambda,\pi_\lambda,$ ect. for $\lambda<v$. $\{pt\}$ denotes the trival system.
\end{thm}

 Then it naturally induces the following commutative diagram associated to $N_d(X)$:
 \begin{equation}\label{eq4.1}
 \xymatrix{
    	N_d(X) \ar[d]_{\pi_0^{(d)}}  & \ldots N_d(X_v) \ar[dr]^{\sigma_{v+1}^{(d)}}\ar[d]_{\pi_v^{(d)}}\ar[l] & &N_d(X_{v+1}) \ldots \ar[ll]_{\phi_{v+1}^{(d)}}\ar[d]_{\pi_{v+1}^{(d)}} & N_d(X_\eta) \ar[d]_{\pi_\eta^{(d)}} \ar[l] \\
    	N_d(Y_0) & \ldots N_d(Y_v) \ar[l]& N_d(Z_{v+1}) \ar[l]^{\rho_v^{(d)}} &N_d(Y_{v+1}) \ldots \ar[l]^{\psi_{v+1}^{(d)}} &N_d(Y_\eta)  \ar[l]
    }
   \end{equation}

 Our main aim in this section is to show the following theorem which will imply Theorem C.

 \begin{thm}
 Let $(X,T)$ be a minimal system. Then $(N_d(X),\G_d)$ has the same structure theorem as $(X,T)$. Precisely,
 in the commutative diagram \ref{eq4.1},  $v\leq\eta$, $\pi_v^{(d)}$ is {\bf RIC}, $\phi_v^{(d)}$
 and $\psi_v^{(d)}$ are proximal, $\rho_v^{(d)}$ are equicontinuous. $\pi_\eta^{(d)}$ is {\bf RIC} and weakly mixing.
 \end{thm}

\subsection{PI extension}
\begin{lem}\label{trivialextension}
	If $\pi:X\lra Y$ is a non-trivial extension, then so is $\pi^{(d)}:N_d(X)\lra N_d(Y)$ for any $d\in \N$.
\end{lem}
\begin{proof}
Since $\pi$ is non-trivial, there is some $y\in Y$ with $|\pi^{-1}(y)|\ge 2$. For each $x\in \pi^{-1}(y)$, we have $x^{(d)}\in N_d(X)$
	and $\pi^{(d)}(x^{(d)})=y^{(d)}$. Thus, $|(\pi^{(d)})^{-1}(y^{(d)})|\ge |\pi^{-1}(y)|\ge 2$. Hence $\pi^{(d)}$ is also non-trivial.
\end{proof}

\begin{lem}\label{double-minimal} Suppose that $\pi:X\lra Y$ is an almost 1-1 extension
(resp. proximal, equicontinuous, distal) between minimal systems.
Then $\pi^{(d)}:N_d(X)\lra N_d(Y)$ is also almost 1-1 (resp. proximal, equicontinuous, distal) for any $d\in\N$.
\end{lem}

\begin{proof} Suppose that $\pi$ is an almost 1-1 extension.
Then there is $y\in Y$ such that $\pi^{-1}(y)=\{x\}$. Thus $(\pi^{(d)})^{-1}(y^{(d)})=\{x^{(d)}\}$.
Hence $\pi^{(d)}:N_d(X)\lra N_d(Y)$ is also an almost 1-1 extension.

Now suppose that $\pi$ is equicontinuous or distal,  then so is $\pi^{(d)}$ just by the definitions.
Suppose that $\pi$ is proximal. It follows from the proof of Lemma \ref{liftproximal} by replacing ${\bf P}(X)$ by $R_{\pi}$.
\end{proof}

\begin{lem}\label{invlim}
Let $\lambda$ be a limit ordinal and $(X_\mu, T)_{\mu\leq \lambda}$ be a collection of minimal systems. If $(X_{\lambda}, T)=\underset{\longleftarrow}{\lim}(X_{\mu}, T)_{\mu<\lambda}$, then for any $d\in N$, $$(N_d(X_{\lambda}), \G_d(T))=\underset{\longleftarrow}{\lim}(N_d(X_{\mu}), \G_d(T))_{\mu<\lambda}.$$
\end{lem}
\begin{proof}
Let $\phi_{\mu,\nu}: X_{\nu}\ra X_{\mu}$ be the homomorphisms associated in the inverse limit for any ordinals $\mu\leq\nu<\lambda$ and let  $\phi_{\mu,\nu}^{(d)}: N_d(X_{\nu})\ra N_d(X_{\mu})$ be the naturally induced homomorphisms.
Then $$X_{\lambda}=\left\{(x_{\alpha})\in\prod_{\alpha<\lambda}X_{\alpha}: \forall  \mu\leq\nu\leq\xi<\lambda, \phi_{\mu,\xi}(x_{\xi})=\phi_{\mu,\nu}\circ\phi_{\nu,\xi}(x_{\xi}) \right\}.$$
By identifying $(\prod_{\mu<\lambda}X_{\mu})^{d}$ with $\prod_{\mu<\lambda}X_{\mu}^{d}$, we write the point of $(\prod_{\mu<\lambda}X_{\mu})^{d}$ in the form of $(x_{\mu}^{1},\cdots,x_{\mu}^d)_{\mu<\lambda}$. If $(x_{\mu}^{1},\cdots, x_{\mu}^d)_{\mu<\lambda}\in N_{d}(X_{\lambda})$,  then there is $(x_{\mu})_{\mu<\lambda}\in X_{\lambda}$ and  sequences $(k_i)$ and $(l_i)$ of integers such that
$$(x_{\mu}^{1},\cdots, x_{\mu}^d)_{\mu<\lambda}=\lim_{i\ra \infty}(T^{l_i+k_i}x_{\mu},T^{l_i+2k_i}x_{\mu},\cdots,T^{l_i+dk_i}x_{\mu})_{\mu<\lambda} . $$
For any $\mu\leq\nu\leq\xi<\lambda$,
\begin{eqnarray*}
\phi_{\mu,\xi}^{(d)} \left((x_{\xi}^{1},\cdots,x_{\xi}^d)\right)&=& \lim_{i\ra \infty}\left(T^{l_i+k_i}\phi_{\mu,\xi}(x_{\xi}),\cdots,T^{l_i+dk_i}\phi_{\mu,\xi}(x_{\xi}))\right) \\
&=& \lim_{i\ra \infty}\left(T^{l_i+k_i}\phi_{\mu,\nu} \circ\phi_{\nu,\xi}(x_{\xi}),\cdots,T^{l_i+dk_i}\phi_{\mu,\nu}\circ\phi_{\nu,\xi}(x_{\xi}))\right) \\
&=& \phi_{\mu,\nu}^{(d)} \circ\phi_{\nu,\xi}^{(d)}\left((x_{\xi}^{1},\cdots,x_{\xi}^d)\right),
\end{eqnarray*}
and $ (x_{\mu}^{1},\cdots, x_{\mu}^d)\in N_{d}(X_{\mu})$. Let
\begin{eqnarray*}
E=  &\Big\{(x_{\mu}^{1},\cdots, x_{\mu}^d)_{\mu<\lambda}\in\prod_{\mu<\lambda}N_{d}(X_{\mu}): \forall  \mu\leq\nu\leq\xi<\lambda, \\
&\phi_{\mu,\xi}^{(d)}\left((x_{\xi}^{1},\cdots,x_{\xi}^d)\right)=\phi_{\mu,\nu}^{(d)}\circ\phi_{\nu,\xi}^{(d)}\left((x_{\xi}^{1},\cdots,x_{\xi}^d)\right) \Big\}.
\end{eqnarray*}
Then $N_d(X_{\lambda})\subset E$. Note that $N_d(X_{\lambda})$ is minimal under the action of
$\G_{d}(\prod_{\mu<\lambda}T)$ and $E$ is minimal under the action of $\prod_{\mu<\lambda}\G_d(T)\cong \G_{d}(\prod_{\mu<\lambda}T)$. It follows that $N_d(X_{\lambda})= E$. Hence $(N_d(X_{\lambda}), \G_d(T))$ is the inverse limit of $(N_d(X_{\mu}), \G_d(T))_{\mu<\lambda}$.
\end{proof}

\begin{thm}  If a minimal system $(X,T)$ is distal (resp. {\bf PI, HPI}), then so is $N_d(T)$.
\end{thm}
\begin{proof}
It follows directly by the definitions and Lemma \ref{trivialextension}, \ref{double-minimal}, \ref{invlim}.
\end{proof}

\subsection{RIC extension}

In this section, we show that if $\phi: (X,T)\ra (Y,T)$ is a {\bf RIC} extension (resp. {\bf RIC} weakly mixing) between minimal systems, then so is $\phi^{(d)}$ for any $d\in N$.  We need the following lemma.

\begin{lem}\label{appi-Nd} Let $(X,T)$ be a minimal system and $d\in\N$. For $(x_1,\ldots,x_d), (y_1,\ldots,y_d)\in N_d(T)$,
if $(x_1,\ldots,x_d, y_1,\ldots,y_d)$ is a minimal point in $(X^{2d},T)$ under the diagonal action, then
$((x_1,\ldots,x_d), (y_1,\ldots,y_d))$ is a minimal point in $(N_d(T)\times N_d(T), \G_{d}(T))$  under the diagonal action.

Generally, let $n, d\in\N$. For $(x_i^1)_{1}^d, \ldots, (x_i^n)_{1}^d\in N_d(T)$,
if $((x_i^1)_{1}^d, \ldots, (x_i^n)_{1}^d)$ is a minimal point in $(X^{n}, T)$ under the diagonal action, then
$ ((x_i^1)_{1}^d, \ldots, (x_i^n)_{1}^d)$ is a minimal point in $ ((N_{d}(T))^{n}, \G_{d}(T))$ under the diagonal action.
\end{lem}

\begin{proof} It is clear for $d=1$. For $d=2$, suppose that $x=(x_1,x_2,y_1,y_2)$ is $S=T\times T\times T\times T$- minimal.
Let $X_4$ be the orbit closure of $x$ under the action of $S$. Then $(X_4,S)$ is minimal, and hence
$(x,x)$ is a minimal point under the action of $\langle S\times S, S\times S^2\rangle$.
That is, $((x_1,x_2,y_1,y_2),(x_1,x_2,y_1,y_2))$ is minimal under the action of
$$\{T^{n+m}\times T^{n+m}\times T^{n+m}\times T^{n+m}\times T^{n+2m}\times T^{n+2m}\times T^{n+2m}\times T^{n+2m}:n, m\in\Z\}.$$
By  projecting to the first, third, sixth and eighth coordinates,  it follows that $(x_1,x_2,y_1,y_2)$ is minimal under the action $\{T^{n+m}\times T^{n+2m}\times T^{n+m}\times T^{n+2m}:n,m\in\Z\}$.

\medskip
For $d>2$, assume that $x=(x_1,\ldots,x_d,y_1,\ldots,y_d)$ is $S=\sigma_{2d}(T)$-minimal.
Let $X_{2d}$ be the orbit closure of $x$ under $S$. Then $(X_{2d},S)$ is minimal, and thus
$x^{(d)}$ is a minimal point for $\langle S\times S\times\ldots\times S, S\times S^2\times\ldots\times S^d\rangle$.
That is $(x_1,\ldots,x_d,y_1,\ldots,y_d)^{(d)}$ is minimal under the action of
$\{S^{n+m}\times S^{n+2m}\times\ldots\times S^{n+dm}:n,m\in\Z\}$. By projecting to the $(k+(k-1)d)^{{\rm th}}$ and $(k+kd)^{{\rm th}}$ coordinates for $k\in\{1,\cdots,d\}$, it follows that $(x_1,\ldots,x_d,y_1,\ldots,y_d)$ is minimal under the action of
$$\{T^{n+m}\times T^{n+2m}\times\ldots\times T^{n+dm}\times T^{n+m}\times T^{n+2m}\times\ldots\times T^{n+dm}:n,m\in\Z\}.$$

\medskip
In general, let $n, d\in\N$ and $(x_i^1)_{1}^d, \ldots, (x_i^n)_{1}^d\in N_d(X)$.
If $x=((x_i^1)_{1}^d, \ldots, (x_i^n)_{1}^d)$ is $ S=\sigma_{nd}(T)$- minimal, then $(X_{nd},S)$ is minimal, where $X_{nd}$ denotes the orbit closure of $x$ under $S$. Thus
$x^{(d)}$ is a minimal point under the action of $\langle S\times S\times\ldots\times S, S\times S^2\times\ldots\times S^d\rangle$.
That is $((x_i^1)_{1}^d, \ldots, (x_i^n)_{1}^d)^{(d)}$ is minimal under the action of
$$\{S^{n+m}\times S^{n+2m}\times\ldots\times S^{n+dm}:n,m\in\Z\}.$$
By projecting to the $(k+(k+j-1)d)^{{\rm th}}$ coordinates for $k,j\in\{1,\cdots,d\}$,
it follows that $((x_i^1)_{1}^d, \ldots, (x_i^n)_{1}^d)$ is minimal under the action of
$$\{(T^{n+m}\times T^{n+2m}\times\ldots\times T^{n+dm})^{(n)}:n,m\in\Z\}.$$
Then $((x_i^1)_{1}^d, \ldots, (x_i^n)_{1}^d)$ is minimal in $((N_d(T))^n,\G_d)$ under the diagonal action.
\end{proof}

\begin{lem}\cite[A.8]{Vr}\label{cts in hyper}
Let $f: X\lra Y$ be a  continuous surjective map between compact metric spaces. Then
\begin{enumerate}
\item the map $F:Y\lra 2^{X}, y\mapsto f^{-1}(\{y\})$ is upper semi-continuous;
\item $F$ is continuous if and only if $f$ is open.
\end{enumerate}
\end{lem}

\begin{lem}\cite{V83}\label{high trans}  
Let $\pi: (X,T)\lra (Y,T)$ be a {\bf RIC} weakly mixing extension. Then $(R_{\pi}^{n},T^{(n)})$ is
topologically transitive for any $n\in \N$.
\end{lem}

\begin{thm}\label{RIClift} Let $\pi:(X,T)\ra (Y,T)$ be a {\bf RIC} (resp. {\bf RIC}  weakly mixing)
extension between minimal systems with $X_{\infty}$ being a factor of $Y$. Then for each $d\in\N$,
$\pi^{(d)}:N_d(X)\lra N_d(Y)$ is also  {\bf RIC} (resp. {\bf RIC} and weakly mixing).
\end{thm}
\begin{proof} \noindent{\bf Claim 1.} For any $d\in\N$, we have $$N_{d}(T)\supset R_{\pi}^{d}=\{(x_1,\cdots,x_d)\in X^d: \pi(x_1)=\cdots=\pi(x_d)\}.$$
\begin{proof}[Proof of Claim]  It is equivalent to show that for any $y\in Y$, $\pi^{-1}(y)\times\cdots\times\pi^{-1}(y)\subset N_d(X)$.  By Lemma \ref{cts in hyper}, the map
$$\Pi^{(d)}: Y^{d}\lra 2^{X^{d}}, (y_1,\cdots,y_d)\mapsto \pi^{-1}(y_1)\times\cdots\times \pi^{-1}(y_d)$$
 is continuous since $\pi$ is open. By  Theorem \ref{lab55}, $Y$ is a $d$-step characteristic factor of $X$, which means there is a dense $G_{\delta}$ subset $\Omega$ of $X$ such that for any $x\in \Omega$,  $L_{x}=\overline{\O}(x, \tau_{d}(T))$ is $\pi^{(d)}$-saturated. Thus for any $y'\in \pi(\Omega)$, $\pi^{-1}(y')\times\cdots\times\pi^{-1}(y')\subset N_d(X)$. Note that $\pi(\Omega)$ is dense in $Y$. Then for any $y\in Y$, there exists a sequence $(y_n)$ in $\pi(\Omega)$ with $y_n\ra y$. By the continuity of $\Pi^{(d)}$, $\pi^{-1}(y_n)\times\cdots\times\pi^{-1}(y_n) $ converges to $\pi^{-1}(y)\times\cdots\times\pi^{-1}(y)$ in the hyperspace. From the definition of the convergence in hyperspaces, we conclude that $\pi^{-1}(y)\times\cdots\times\pi^{-1}(y)\subset N_d(X)$. This proves the claim.
\end{proof}

Now Fix $d\in \N$ and let $\phi=\pi^{(d)}:N_d(X)\lra N_d(Y)$.

\medskip
By the definition of {\bf RIC} extensions, it follows that $\pi$ is open and hence  $\phi$ is open. It remains to show that for any $n\in \N$, the minimal points in $R_{\phi}^n$ are dense, where
$$R_\phi^n=\left\{\left((x_i^1)_{1}^{d},\cdots, (x_i^n)_{1}^{d}\right)\in N_d(X)^n: ~~\phi\left((x_i^1)_{1}^{d}\right)=\cdots= \phi\left((x_i^n)_{1}^{d}\right)\right\}.$$
Let $M$ denote the set of minimal points in $ R_\phi^n$ under the diagonal action of $\G_d(T)$ and
$N$ denote the set of minimal points in $R_{\pi}^{nd}$ under $T^{(nd)}$. According to Lemma \ref{appi-Nd} and Claim 1, we have

\noindent{\bf Claim 2.} $N\subset M$.

Thus for any $y\in Y$ and $\mathbf{x}\in \phi^{(-n)}(y^{(nd)})$, $\mathbf{x}\in \overline{M}$. Given $\left((x_i^1)_{1}^{d},\cdots, (x_i^n)_{1}^{d}\right)\in R_\phi^n$, it follows that
$$ \pi(x_i^1)=\pi(x_i^2)\ldots=\pi(x_i^n)=y_i, \  \ i=1,2,\cdots,d, $$
for some $(y_1,y_2,\ldots,y_d)\in N_d(Y)$.  By the definition of $N_{d}$, there exist a sequence $(g_n)$ in  $\G_{d}(T)$ and $u_n\in Y$ with $g_n(u_n,\cdots, u_n) \lra (y_1,\cdots,y_d)$. Since $\Pi^{(d)}$ is continuous, there exists $(w_i^1)_{1}^{d},\cdots, (w_{i}^{n})_{1}^{d}\in \phi^{(-1)}(u_{n}^{(d)})$ satisfying
$$g_n(w_i^1)_{1}^{d}\lra (x_i^1)_{1}^{d}, \cdots,  g_n(w_i^n)_{1}^{d}\lra (x_i^n)_{1}^{d}.$$
Thus $\left((x_i^1)_{1}^{d},\cdots, (x_i^n)_{1}^{d}\right)\in\overline{M}$ by noting that $(g_n(w_i^1)_{1}^{d}, \cdots,  g_n(w_i^n)_{1}^{d})\in\overline{M}$. Hence $M$ is dense in $R_{\phi}^{n}$.

\bigskip
Now we discuss {\bf RIC} weakly mixing extensions. It remains to show that $\phi$ is weakly mixing.
Since $\pi$ is {\bf RIC} weakly mixing, $(R_{\pi}^{2d},T^{(2d)})$ is transitive by Lemma \ref{high trans}. Then there exists $y\in Y$and $z_1,\cdots,z_{2d}\in \pi^{-1}(y)$  such that
$$R_{\pi}^{2d}=\overline{\O}((z_1,\cdots,z_{2d}), T^{(2d)}). $$

 Suppose that $((x_1,\cdots,x_d),(x_1',\cdots,x_d'))\in R_{\phi}$.  Then we have $\pi(x_i)=\pi(x_i')=y_i$ for some $y_i\in Y$ and $(y_1,\cdots,y_d)\in N_{d}(Y)$.  There exist a sequence $(g_n)$ in $\G_{d}(T)$ and a sequence $(v_n)$ in $Y$ satisfying $g_n(v_n^{(d)})\lra (y_1,\cdots, y_d)$. By the continuity of $\Pi^{(d)}$, there are sequences $(w_1^n,\cdots, w_d^n), ({w'}_1^n,\cdots, {w'}_d^n) \in \pi^{-1}(u_n)\times\cdots\times \pi^{-1}(u_n)$ with
 $$ g_n(w_1^n,\cdots, w_d^n)\lra (x_1,\cdots,x_d),\  \text{ and }\  g_n({w'}_1^n,\cdots, {w'}_d^n)\lra (x_1',\cdots,x_d') .$$
 Note that $((w_1^n,\cdots, w_d^n),   ({w'}_1^n,\cdots, {w'}_d^n))\in \overline{\O}((z_1,\cdots,z_{2d}), T^{(2d)})$. Thus
 $$((x_1,\cdots,x_d),(x_1',\cdots,x_d'))\in \overline{\O}((z_1,\cdots,z_{2d}), \G_{d}(T)).$$
 Therefore $R_{\phi}$ is topologically transitive under the diagonal action of $\G_{d}(T)$.
\end{proof}

Finally we have the following remark
\begin{rem} In Theorem \ref{RIClift}, we have the assumption that $X_\infty$ is a factor of $Y$. It is clear that
$Y_\eta$ in Theorem \ref{STMS} satisfies this assumption.
\end{rem}

\section{Saturation examples}

In this section we give saturation examples. First we need some simple lemmas.

\begin{lem}\label{simple-lemma} Let $(X,T)$ be minimal, $\pi: X\lra X_{eq}$ be the factor map, and $d\ge 2$.
	For a given $x\in X$, let $(x_1,x_2,\ldots,x_d)\in L_x$. Then for any subsequence $\{n_i\}\subset \Z$
	with $T^{jn_i}x\lra x_j$ for some $1\le j\le d$, we have
	$T^{kn_i}x\lra x_k'$ with $\pi(x_k)=\pi(x_k')$ for any $1\le k\le d.$
\end{lem}
\begin{proof} Let $y=\pi(x)$ and $y_i=\pi(x_i), 1\le i\le d$. Then we have $(y_1,y_2,\ldots,y_d)\in L_y.$
	So, there is $g\in X_{eq}$ such that $y_i=y+ig, 1\le i\le d$. Since $T^{jn_i}x\lra x_j$, we have $T^{jn_i}y\lra y_j=y+jg.$
	Assume that $(T^{n_i}y, T^{2n_i}y,\ldots, T^{dn_i}y)\lra (y+g',y+2g',\ldots, y+dg')$ for some $g'\in X_{eq}$.
	This implies that $g'=g$ and hence
	$$(T^{n_i}y, T^{2n_i}y,\ldots, T^{dn_i}y)\lra (y+g,y+2g,\ldots, y+dg).$$
	As $T^{kn_i}x\lra x_k'$, we get that $T^{kn_i}y\lra \pi(x_k')$ for each $1\le k\le d$. Thus, $\pi(x_k')=y+kg=y_k=\pi(x_k)$
	for each $1\le k\le d.$
	
\end{proof}

For any abelian group $G$, $A\subset G$ and a rational number $r=\frac{p}{q}$ we define
$$rA=\{g\in G: qg=pa\ \text{for some}\ a\in A\}.$$

With above lemmas we can show

\begin{thm}Let $(X,T)$ be minimal, $d\ge 2$ and $\pi:X\lra Y=X_{eq}$ be almost 1-1. Let $A=\{y\in X_{eq}, |\pi^{-1}y|\ge 2\}$.
	We have
	\begin{enumerate}
		\item If $2A-A$ is also first category, then $\pi$ is saturated for $T\times T^2$.
		\item If $B_d=:\bigcup_{1\le i\le j\le d}(\frac{j}{j-i}A-\frac{i}{j-i}A)$ is also first category,
		then $\pi$ is saturated for $T\times T^2\times \ldots \times T^d$ and $d\ge 3$.
	\end{enumerate}
\end{thm}
\begin{proof} First we consider the case when $d=2$.
	
	Let $x\in X$ and $L_x={\overline{O}}((x,x),T\times T^2)$. Let $y=\pi(x)$. Then
	$$L_y=\{(y+g,y+2g):g\in Y\}.$$
	If $y+g=y_1\in A$ and $y+2g=y_2\in A$, then $y=2y_1-y_2\in 2A-A$.
	
	Let $\Omega=\pi^{-1}(2A-A)^c$. Then $\Omega$ is a dense $G_\delta$ set. We are going to show that
	for each $x\in \Omega$, $L_x$ is saturated. Let $y=\pi(x)$. For $(x_1,x_2)\in L_x$ we have the following 3 cases.
	
	\medskip
	\noindent {\bf Case (1):} $\pi(x_1),\pi(x_2)\in (2A-A)^c$.
	
	\medskip
	
	In this case, $|\pi^{-1}\pi(x_1)|=1$ and $|\pi^{-1}\pi(x_2)|=1$.
	It is clear that $\pi^{-1}\pi(x_1)\times \pi^{-1}\pi(x_2)\in L_x$.
	
	\medskip
	
	\noindent {\bf Case (2):} $\pi(x_1)\in 2A-2A,\pi(x_2)\in (2A-A)^c$.
	
	\medskip
	
	In this case we may assume that
	$|\pi^{-1}\pi(x_1)|\ge 2$ and $|\pi^{-1}\pi(x_2)|=1$. For any $z\in \pi^{-1}\pi(x_1)$, there is a subsequence $\{n_i\}$ of $\Z$
	with $T^{n_i}x\lra z$. It is clear that $T^{2n_i}x\lra x_2$ by Lemma \ref{simple-lemma}. This implies that $\pi^{-1}\pi(x_1)\times \pi^{-1}\pi(x_2)\in L_x$.
	
	\medskip
	
	\noindent {\bf Case (3):} $\pi(x_1)\in (2A-A)^c,\pi(x_2)\in (2A-A)$.
	
	\medskip
	
	In this case we may assume that
	$|\pi^{-1}\pi(x_1)|=1$ and $|\pi^{-1}\pi(x_2)|\ge 2$. First we note that there is
	$z\in \pi^{-1}\pi(x_2)$ such that $(x_1,z)\in L_x$. This implies that there is a subsequence $\{n_i\}$ of $\Z$ with
	$T^{n_i}x\lra x_1$ and $T^{2n_i}x\lra z$.
	This implies that $x,z\in X_1$, where $X_1$ is minimal under $T^2$.
	Since $(X,T)$ is almost 1-1, and hence $\pi$ is proximal. By Lemma \ref{splits-minimal}, $\pi^{-1}\pi(x_2)\subset X_1$.

	
	Thus, for each $z'\in \pi^{-1}\pi(x_2)$ there is a subsequence $\{m_i\}$ of $\Z$ such that $T^{2m_i}x\lra z'$.
	Then we have $T^{m_i}x\lra x_1$ by Lemma \ref{simple-lemma}. This implies that $\pi^{-1}\pi(x_1)\times \pi^{-1}\pi(x_2)\in L_x$.
	
	\medskip
	To sum up we have proved that for each $x\in \Omega$, $L_x$ is saturated for $T\times T^2$.
	
	\medskip
	For the general case when $d\ge 3$ and $x\in X$, let $L_x={\overline{O}}((x,\ldots,x),T\times, \ldots,\times T^d)$. Let $y=\pi(x)$. Then
	$$L_y=\{(y+g,y+2g, \ldots, y+dg):g\in Y\}.$$
	
	If $y+ig=y_1\in A$ and $y+jg=y_2\in A$ with$1\le i<j\le d$, then $(j-i)g=y_2-y_1$, and so, $y=\frac{j}{j-i}y_1-\frac{i}{j-i}y_2 \in \frac{j}{j-i}A-\frac{i}{j-i}A$.
	
	Since $B_d$ is also of first category,
	$\Omega=\pi^{-1}B^c$ is a dense $G_\delta$ subset of $X$ and the same proof is applied,
	since for any $x\in \Omega$ and $g\in Y$, there exists at most one $1\le j\le d$ such that $\pi(x)+jg\in A$.
	
	\medskip
	To sum up we have proved that $L_x$ is saturated for each $x\in \Omega$ under $T\times \ldots\times T^d$.
\end{proof}

\begin{rem}By what we have proved we know that Denjoy, Floyad examples are saturated
	since for such systems, since $B_d$ is countable.
\end{rem}




\section{The construction of a non-saturated example}

Our aim in this section is to construct a non-saturated minimal system which is a proximal extension of its
maximal equicontinuous factor. The following theorem tells us that it suffices to find counterexamples in almost automorphic systems.
\begin{thm} Assume that $(X,T)$ is minimal and $\pi:X\lra X_{eq}$ is a proximal extension.
	If it is not $\pi$-saturated, then there is $\tau^*: X_{eq}^*\lra X_{eq}$ which is almost 1-1 and is not
	$\tau^*$-saturated, where $X_{eq}^*$ is minimal.
\end{thm}
\begin{proof} If $\pi$ is almost one to one, we are done. Thus, we may assume that $\pi$ is not almost one to one.
	It is clear that $\pi$ is not open by \cite{GHSWY}.
	
	Now there are minimal systems $X^*$ and $X_{eq}^*$, and factor maps $\pi^*:X^*\lra X_{eq}^*$, $\sigma^*:X^*\lra X$
	and $\tau^*:X_{eq}^*\lra X_{eq}$ such that $\pi^*$ is open and $\sigma^*,\tau^*$ are almost 1-1 with $\tau^*\circ \pi^*=\pi\circ\sigma^*$.
	\[
	\begin{CD}
		X @<{\sigma^*}<< X^*\\
		@VV{\pi}V      @VV{\pi^*}V\\
		X_{eq} @<{\tau^*}<< X_{eq}^*
	\end{CD}
	\]
	
	By \cite{GHSWY}, $\pi^*$ is saturated. We are going to show that $\tau^*$ is not saturated.
	
	Assume the contrary that $\tau^*$ is saturated. Thus, there is a dense $G_\delta$ set $\Omega_1$ of $X^*$ such that
	for each $x^*\in \Omega_1$, $L_{x^*}$ is $\pi^*$-saturated, and there is a dense $G_\delta$ set $\Omega_2$ of $X_{eq}^*$ such that
	for each $x_{eq}^*\in \Omega_2$, $L_{x_{eq}^*}$ is $\pi^*$-saturated. Let $\Omega_3=\Omega_1\cap (\pi^*)^{-1}\Omega_2$
	and $\Omega=\sigma^*\Omega_3$. It is clear that $\Omega$ is a dense $G_\delta$ set of $X$.
	
	For $x\in \Omega$ pick $x^*\in \Omega_3$ with $\sigma^*(x^*)=x$.
	It is clear that $L_{x^*}$ is $\pi^*$-saturated and $\pi^*(L_{x^*})=L_{\pi^*x^*}$ is $\tau^*$-saturated.
	This implies that $L_{x^*}$ is $\tau^*\circ \pi^*$-saturated by Lemma \ref{lem-saturated-basic}. Thus,
	$L_x=\sigma^*L_{x^*}$ is $\pi$-saturated again by Lemma \ref{lem-saturated-basic}, a contradiction.
	
\end{proof}

In the following we will construct a minimal system $(X,T)$ which is an almost one to one extension of an equicontinuous system but not saturated for $T\times T^2$. First we need some lemmas.

\subsection{Some lemmas}
\begin{lem}\cite[Proposition 2.41]{Foll}\label{continuity of convolution}
	Let $G$ be a locally compact abelian group. If $f\in L^p(G), g\in L^q(G)$, where $1<p<\infty$ and $\frac{1}{p}+\frac{1}{q}=1$, then
	the convolution $f*g\in C_c(G)$.
\end{lem}

\begin{lem}
	Let $G$ be a locally compact abelian group and let $\mu$ be the Haar measure of $G$. Suppose that $A,B$ are compact subsets of $G$.
	If $\mu(A)>0$ and $\mu(B)>0$, then $A+B$ contains a nonempty open subset.
\end{lem}
\begin{proof}
	For every $x\in G$,
	\begin{eqnarray*}
		\chi_A*\chi_B(x)&=&\int\chi_A(x-y)\chi_B(y)dy\\
		&=&\int\chi_{x-A}(y)\chi_B(y)dy\\
		&=&\int\chi_{(x-A)\cap B}(y)dy\\
		&=&\mu((x-A)\cap B).
	\end{eqnarray*}
	Thus
	\[\int \chi_A*\chi_B(x)dx=\int\int\chi_A(x-y)\chi_B(y)dydx=\mu(A)\mu(B)>0.\]
	
	Let $E=\{x\in G: \chi_A*\chi_B(x)>0\}=\{x\in G: \mu((x-A)\cap B)>0\}$. Note that for every $x\in E$, $(x-A)\cap B\neq \emptyset$ which is
	equivalent to $x\in A+B$. Thus $E\subset A+B$. By Lemma \ref{continuity of convolution}, $\chi_A*\chi_B$ is continuous. Therefore, $E$ is a nonempty open subsets contained in $A+B$.
\end{proof}

\begin{lem}\label{2times} Let $G$ be a compact metrizable  monothetic group and let
	$\mu$ be the Haar measure of $G$. If $A\subset G$ with positive measure, then $\mu(nA)>0$ for any $n\in\Z\setminus\{0\}$.
\end{lem}
\begin{proof} Since $G$ is monothetic, it is abelian.
	Note that $\mu(-nA)=\mu(nA)$ by the uniqueness of the Haar measure. So, it remains to show
	the theorem for $n\ge 2$. We show the case when $n=2$, and the proof for the rest $n>2$ is similar.
	
	\medskip
	Assume that $g_0\in G$ with $\{ng_0:n\in\Z\}$ is dense in $G$. Let $G_1=\{2ng_0:n\in\Z\}$, then $G=G_1\cup (G_1+g_0)$.
	There are two cases: $G_1=G$ or $G_1\not=G$.
	
	In the first case we let $\phi:G\lra G$ with $g\mapsto 2g$ and $\nu=\phi_*\mu$. And in the second case
	let $\phi:G\lra G_1$ with $g\mapsto 2g$ and $\nu=\phi_*\mu$.
	It is clear that $\phi$ is surjective.
	
	In the first case for any $B\subset G$ and $g\in G$ we claim that
	$$\phi^{-1}(B+g)=\phi^{-1}B+g'\ \text{with}\ 2g'=g.$$
	(Note that the existence of $g'$ follows from the subjectivity of $\phi$.)
	To see this, let $y\in \phi^{-1}(B+g)$. Then $2y\in B+g$ which implies that $2(y-g')\in B$. Thus, $y-g'\in \phi^{-1}B$, i.e.
	$y\in \phi^{-1}B+g'$. The converse is true by the same reason.
	
	This shows that $\nu=\mu$. Thus, $\mu(2A)=\mu(\phi^{-1}(2A))=\mu(A+Ker \phi)\ge \mu(A)>0$.
	
	In the second case for any $B\subset G_1$ we claim that $\phi^{-1}(B+g)=\phi^{-1}B+g'$ with $2g'=g$.
	This shows that $\nu$ is the Haar measure $\mu_1$ on $G_1$. Without loss of generality we assume that
	$A\subset G_1$, otherwise we replace $A$ by $A+g_0$ which has the same measure as $A$.
	
	Thus, we have that $\mu_1(2A)=\mu(\phi^{-1}(2A))\ge \mu(A)>0$.
	It deduces that $\mu(2A)\ge \frac{1}{2}\mu_1(2A)>0$, as $\mu=\frac{1}{2}\mu_1+ \frac{1}{2}\mu_2,$ where $\mu_2=\varphi_*\mu_1$
	with $\varphi: G_1\lra G_1+g_0$, $g\mapsto g+g_0$.
	
	
	
	
\end{proof}

\begin{rem} We remark that
	
	\begin{enumerate}
		\item The monothetic assumption is essential. $G=\Z_2\times \Z_2\times \ldots$ is a counterexample.
		
		\item For the adding machine $X$ and $y\in X$, it is not true that there $x\in X$ with $2x=y$.
		For example, let $a=(1,0,\ldots)$. Let $b=(x_1,x_2,\ldots)\in X$
		then $2b=(0,y_2,\ldots)$. So, there is no $b\in X$ with $2b=a$. This indicates the second case can happen.
		
		\item Lemma \ref{2times} can be proved by using the structure theorem of a compact metrizable abelian monothetic group.
		Such a group is an inverse limit of groups $G_i=\mathbb{T}^{n_i}\times \Z_{p_i}.$ So, the problem can be deduced
		to the same one on $G_i$. This can be done by a direct computation.
	\end{enumerate}
\end{rem}

\begin{cor}\label{positiveU} Let $G$ be a compact metrizable abelian monothetic group and let $\mu$ be the Haar measure of $G$.
	If $A\subset G$ with $\mu(A)>0$, then $2A-A$ contains a non-empty open set $U$.
\end{cor}

\subsection{Toeplitz flows}
Let $\Sigma$ be a finite set and $X=\Sigma^{\mathbb{Z}}$. Then $X$ is a compact metric space with a metric $d$ defined by
\[ \rho(x,y)=\sum_{n=-\infty}^{\infty}\frac{|x(n)-y(n)|}{2^{|n|}}.\]
Let $T$ denote the left shift on $X$, i.e. $Tx(n)=x(n+1)$ for all $x\in X$ and $n\in\Z$.

For  and $x\in X$, $\sigma\in \Sigma$ and a positive integer $p\in\mathbb{N}^*$,  set
\begin{eqnarray*}
	&&\text{Per}_p(x, \sigma)=\{n\in\mathbb{Z}:~x(n')=x(n)=\sigma ~~\text{for all }~~n'\equiv n\mod p\},\\
	&&\text{Per}_p(x)=\cup_{\sigma\in \Sigma}\text{Per}_p(x, \sigma),\\
	&&\text{Aper}(x)=\mathbb{Z}\setminus(\cup_{p\in\mathbb{N}*}\text{Per}_p(x)).
\end{eqnarray*}
A sequence $x\in X$ is called a \textit{Toeplitz sequence} if $\text{Aper}(x)=\emptyset$.

By the $p-skeleton$ of $x\in X$ we mean the part of $x$ which is periodic with period $p$. We call that $p$ is an
{\it essential period} of $x$ if the $p$-skeleton of $x$ is not periodic with any smaller period.

\begin{de}
	A \textit{periodic structure} for a non-periodic Toeplitz sequence $\eta$ is an increasing sequence $(p_i)_{i\in\mathbb{N}^*}$
	of positive integers satisfying
	\begin{itemize}
		\item [(1)] $p_i$ is an essential period of $\eta$ for all $i\in\N$,
		\item [(2)] $p_i | p_{i+1} $ for all $i$,
		\item [(3)] $\bigcup_{i=1}^{\infty}\text{Per}_{p_i}(\eta)=\mathbb{Z}$.
	\end{itemize}
\end{de}
Note that every non-periodic Toeplitz sequence has a periodic structure. In the following, we assume that $\eta$ is a non-periodic
Toeplitz sequence and fix a periodic structure $(p_i)$ of $\eta$. Let $G$ be the inverse limit group
$\underset{\longleftarrow}{\lim}~ \mathbb{Z}/p_i\mathbb{Z}$, i.e.
\[ G=\{(n_i): ~n_i\in \mathbb{Z}/p_i\mathbb{Z}~~\text{and}~~n_j\equiv n_i\mod p_i~~~\text{for all} ~~i<j\}.\]
Let $\hat{1}$ denote the element $(1)$ in $G$ and $\hat{n}=n\cdot \hat{1}$ for $n\in\mathbb{Z}$. Then $G$ a compact monothetic group generated by $\hat{1}$.

For each $i\in\mathbb{N}^*$, $n\in \mathbb{Z}/p_i\mathbb{Z}$, set
\[A_{n}^i=\{ T^{m}\eta:~m\equiv n\mod p_i\}.\]
\begin{lem} \cite[Lemma 2.3]{Williams}
We have
	\begin{enumerate}
		\item [(1)] $\overline{A_{n}^i}$ is exactly the set of all $\omega\in \overline{\mathcal{O}}(\eta)$ with the
		same $p_i$-skeleton as $T^{n}\eta$;
		\item [(2)] $ \{ \overline{A_{n}^i}:~n\in \mathbb{Z}/p_i\mathbb{Z}\}$ is a partition of
		$\overline{\mathcal{O}}(\eta)$ into relatively open sets;
		\item [(3)]  $  \overline{A_{n}^i}\supset \overline{A_{m}^j}$ for  $i<j$ and $m\equiv n\mod p_i$;
		\item [(4)] $T\overline{A_{n}^i}=\overline{A_{n+1}^i}$.
	\end{enumerate}
\end{lem}

For $g=(n_i)\in G$, set $A_g=\bigcap_{i=1}^{\infty}\overline{A_{n_i}^i}$. It is clear that $A_g$ is nonempty.
Define a map $\pi: \overline{\mathcal{O}}\rightarrow G$ by $\pi^{-1}(g)=A_g$.
\begin{thm} \cite{Williams}
We have
	\begin{enumerate}
		\item [(1)] $(G,\hat{1})$ is the maximal equicontinuous factor of $(\overline{\mathcal{O}},T)$ with the factor map $\pi$.
		\item [(2)] $\pi(\omega)=\pi(\omega')$ if and only if $\omega$ and $\omega'$ have the same $p_i$-skeleton, for all $i\in\mathbb{N}^*$. In particular, $\pi$ is one-to-one on the set of Toeplitz sequences in $\overline{\mathcal{O}}$ .
	\end{enumerate}
\end{thm}

Since $\text{Per}_{p_i}(\eta)$ is periodic, it has a density in $\mathbb{Z}$ given by
\[ d_i=\frac{1}{p_i}\cdot \#\left\{ n\in \in \mathbb{Z}/p_i\mathbb{Z}:~n\in\text{Per}_{p_i}(\eta)\right\}.\]
Then $(d_i)$ is increasing. Set $d=\lim\limits_{i\rightarrow\infty}d_i$. The Toeplitz sequence $\eta$ is called \textit{regular} if $d=1$.

Let $R=\{g\in G: ~|\pi^{-1}(g)|=1\}$. Then $\pi^{-1}(R)$ consists of Toeplitz sequences in $\overline{\mathcal{O}}$.

\begin{lem}\cite{Williams}\label{meas of reg pts}
	$R$ is measurable and
	\[ m(R)=\begin{cases}
		1 &\text{if}~~~ d=1\\
		0 &\text{if}~~~ d<1,
	\end{cases}\]
	where $m$ denotes the Haar measure on $G$.
\end{lem}

\subsection{The construction}
Let $\Sigma=\{0,1,2,3,4\}$ and let $(q_i) $ be an increasing sequence of even integers satisfying $q_1\geq 6$ and $\sum_{i}\frac{1}{q_i}$ converges. Let $p_0=1$ and $p_i=p_{i-1}q_{i}$, for every $i\in\mathbb{N}^*$. Now we construct a Toeplitz sequence $\eta$ inductively.

\noindent{\bf Step 1}.  Set
$$\eta(n)=\begin{cases}
	0,& n\equiv 0\mod q_{1},\\
	1,& n\equiv 1\mod q_{1},\\
	2,& n\equiv 2\mod q_{1},\\
	3,& n\equiv \frac{q_{1}}{2}+1\mod  q_1,\\
	4,& n\equiv -1\mod q_{1}.
\end{cases}$$
For each $k\in\mathbb{Z}$,  let
$$J(1,k)=[kp_1+3, \frac{q_{1}}{2}]\cup [\frac{q_{1}}{2}+2, (k+1)p_1-2].$$

\noindent {\bf Step 2}.  For  all $n\in J(1,k)$, set
\[\eta(n)=\begin{cases}
	0,& k\equiv 0\mod q_{2},\\
	1,& k\equiv 1\mod q_{2},\\
	2,& k\equiv 2\mod q_{2},\\
	3,& k\equiv\frac{q_{2}}{2}+1\mod q_2,\\
	4,& k\equiv -1\mod q_{2}.
\end{cases}\]
For each $k\in\mathbb{Z}$, let $J(2,k)$ denote the set of $n\in[kp_2+1, (k+1)p_2-1]$ that has not been defined at the end of the second step.

Assume that we have completed the $i^{\text{th}}$ step.  For each $k\in\mathbb{Z}$, let $J(i,k)$ denote the set of $n\in[kp_i+1, (k+1)p_i-1]$ that has not been defined at the end of $i^{\text{th}}$ step.\\

\noindent {\bf Step $i+1$}.  For  all $n\in J(i,k)$, set
\[\eta(n)=\begin{cases}
	0,& k\equiv 0\mod q_{i+1},\\
	1,& k\equiv 1\mod q_{i+1},\\
	2,& k\equiv 2\mod q_{i+1},\\
	3,& k\equiv\frac{q_{i+1}}{2}+1\mod q_{i+1},\\
	4,& k\equiv -1\mod q_{i+1}.
\end{cases}\]

Note that  $\eta$ has been defined on $[-p_{i},3p_{i}]$  at the end of $(i+1)^{\text{th}}$ step and the construction is periodic at each step, so $\eta$ is a Toeplitz sequence. Furthermore, it is clear that $\eta$ is aperiodic.

\begin{lem}
	The sequence $\eta$ is regular if and only if $\sum_{i=1}^{\infty}\frac{1}{q_i}$ diverges.
\end{lem}
\begin{proof}
	Recall that $d_i=\frac{1}{p_i}\cdot \#\{n\in\mathbb{Z}/p_i\mathbb{Z}: n\in \text{Per}_{p_i}(\eta)\}$. Then $d_1=\frac{4}{p_1}$. For $i\geq 1$,
	\[d_{i+1}=d_i+(1-d_i)\frac{4}{q_{i+1}}.\]
	Thus
	\[ 1-d_{i+1}=(1-d_1)\prod_{j=1}^{i}\left(1-\frac{4}{q_{j+1}}\right).\]
	Then $d_i\rightarrow 1$ if and only if $\prod_{j=1}^{i}\left(1-\frac{4}{q_{j+1}}\right)\rightarrow 0$, if and only if $\sum_{i=1}^{\infty}\frac{1}{q_i}$ diverges..
\end{proof}

\begin{lem}\label{Oxtoby seq}\cite[Lemma 3.3]{Williams}
	\begin{enumerate}
		\item [(1)] For all $\omega\in\overline{\mathcal{O}}(\eta)$, $\omega(n)$ is constant on $\text{Aper}(\omega)$.
		\item [(2)] For each $g\in G$ and $\sigma\in \Sigma$, there is an $\omega\in\pi^{-1}(g)$ with $\omega(n)=\sigma$ for all $n\in \text{Aper}(g)$.
	\end{enumerate}
\end{lem}
\begin{proof}
	(1) Let $\pi(\omega)=g=(n_1,n_2,\cdots)$. For each $i$,  $\omega$ has the same $p_i$-skeleton as $T^{n_i}\eta$ and so
	$$[-n_i, p_i-n_i) \cap \text{Aper}(\omega)\subset [-n_i, p_i-n_i)\setminus\text{Per}_{p_i}(T^{n_i}\eta)=J(i,0)-n_i.$$
	If $T^{m}\eta\in A^{i}_{n_i}$ then $m=n_i+kp_{i}$ for some $k\in\mathbb{Z}$ , then $T^{m}\eta$ is constant on $J(i,0)-n_i$ since $\eta$ is constant on $J(i,k)$. Since $\omega\in \overline{A_{n_i}^{i}}$, $\omega$ must also be constant on $J(i,0)-n_i$. Hence $\omega$ is constant on $ [-n_i, p_i-n_i) \cap \text{Aper}(\omega)$ for all $i$.  If $\text{Aper}(\omega)\neq \emptyset$, then $-n_i\rightarrow -\infty$ and $p_{i}-n_i\rightarrow \infty$. Therefore, $\omega$ is constant on $\text{Aper}(\omega)$.
	
	(2) Let $g=(n_i)\in G$ and $\sigma\in \Sigma$. The sequences $T^{n_i}\eta$ all have the same $p_j$-skeleton for $i>j$, so $T^{n_i}\eta(n)$ is eventually constant for each $n\in\text{Aper}(g)$.  For any positive integer $i$,  $T^{n_{i}}\eta(n)=0$ for $n\in J(i,0)-n_{i}$, which contains $[-n_{i}, p_{i}-n_{i}) \cap \text{Aper}(\omega)$;  $T^{n_{i}+p_i}\eta(n)=1$ for $n\in J(i,1)-n_{i}-p_i$, which contains $[-n_{i}, p_{i}-n_{i}) \cap \text{Aper}(\omega)$; $T^{n_{i}+2p_i}\eta(n)=0$ for $n\in J(i,2)-n_{i}-2p_i$, which contains $[-n_{i}, p_{i}-n_{i}) \cap \text{Aper}(\omega)$; $T^{n_{i}+\left(\frac{q_{i+1}}{2}+1\right)p_i}\eta(n)=3$ for $n\in J(i,-1)-n_{i}-\left(\frac{q_{i+1}}{2}+1\right)p_i$, which contains $[-n_{i}, p_{i}-n_{i}) \cap \text{Aper}(\omega)$; $T^{n_{i}-p_i}\eta(n)=4$ for $n\in J(i,-1)-n_{i}+p_i$, which contains $[-n_{i}, p_{i}-n_{i}) \cap \text{Aper}(\omega)$.  Therefore, if $\sigma=k\in\{0,1,2\}$ then  $T^{n_{i}+kp_{i}}\eta$ converges to the desired $\omega$ as $i\rightarrow\infty$, if  $\sigma=3$ then  $T^{n_{i}+\left(\frac{q_{i+1}}{2}+1\right)p_i}\eta$ converges to the desired $\omega$ as $i\rightarrow\infty$ , and   if $\sigma=4$ then  $T^{n_{i}-p_{i}}\eta$ converges to the desired $\omega$ as $i\rightarrow\infty$.
\end{proof}

\begin{lem}\label{min per}
	For a nonnegative integer $n$, if we write
	$ n=n_0+n_1p_1+n_2p_2+n_3p_3+\cdots $, where $n_{i}\in\{0,1,\cdots, q_{i+1}\}$,
	then $j=\min\left\{i: n_i\in\{0,1,2, \frac{q_{i+1}}{2}+1, q_{i+1}-1\}\right\}$ is
	the smallest integer such that $n\in \text{Per}_{p_{j+1}}(\eta)$.
\end{lem}
\begin{proof} From the construction of $\eta$, $n$ is defined for $\eta$ at the $j^{\text{th}}$ step and the lemma is followed.
\end{proof}

\begin{lem}
	For any $g=(g_1,g_2,\cdots)\in G$,  the fiber of $g$ is a singleton if and only if there exists an increasing subsequence $(n_i)$ such that  for each $i$,
	\[ \frac{g_{n_{i}+1}-g_{n_{i}} }{p_{n_i}} \in\left\{ 0,1,2, \frac{q_{n_i+1}}{2}+1, q_{n_i+1}-1 \right\} .\]
\end{lem}
\begin{proof}
	Suppose that $|\pi^{-1}(g)|=1$ and let $x=\pi^{-1}(g)$. Then $x$ is a Toeplitz sequence. Thus for any $i$, there exists $j$ such that $[-g_i, p_i-g_i)\subset\text{Per}_{p_j}(x)$.  It is clear that $j\geq i+1$, otherwise $x$ is periodic. Since $x$ has the same $p_j$-skeleton as $T^{g_j}\eta$, we have $[-g_i, p_i-g_i)\subset\text{Per}_{p_j}(T^{g_j}\eta)$. Hence $ [-g_i, p_i-g_i)+g_j\subset\text{Per}_{p_j}(\eta)$.  For any $n\in[0,p_i)$,
	\begin{eqnarray*}
		g_j-g_i+n&=&n+ (g_{i+1}-g_i)+(g_{i+2}-g_{i+1})+\cdots+(g_j-g_{j-1})\\
		&=&n+\frac{g_{i+1}-g_i}{p_i}p_i+\frac{g_{i+2}-g_{i+1}}{p_{i+1}}p_{i+1}+\cdots+\frac{g_j-g_{j-1}}{p_{j-1}}p_{j-1}.
	\end{eqnarray*}
	Taking $n=3+3p_1+\cdots+3p_{i-1}$, by Lemma \ref{min per}, there exists $k$ with $i+1\leq k\leq j$ such that $\frac{g_{k}-g_{k-1} }{p_{k}} \in\left\{ 0,1,2, \frac{q_{k+1}}{2}+1, q_{k+1}-1 \right\}$. Since $i$ is arbitrary, we can find an increasing subsequence $(n_i)$ such that  for each $i$,
	\[ \frac{g_{n_{i}+1}-g_{n_{i}} }{p_{n_i}} \in\left\{ 0,1,2, \frac{q_{n_i+1}}{2}+1, q_{n_i+1}-1 \right\} .\]
	
	Conversely, suppose that there exists an increasing subsequence $(n_i)$ such that  for each $i$,
	$ \frac{g_{n_{i}+1}-g_{n_{i}} }{p_{n_i}} \in\left\{ 0,1,2, \frac{q_{n_i+1}}{2}+1, q_{n_i+1}-1 \right\} $. We may assume that $-g_i\rightarrow -\infty$ and $p_i-g_i\rightarrow \infty$, otherwise there $g=\widehat{m}$ for some integer $m$ whence $|\pi^{-1}(g)|=1$. For any $k$, there exists $i$ such that $ k< n_{i}$. Then for any $n\in [-g_k, p_k-g_k)$,
	\begin{eqnarray*}
		g_{n_{i}}-g_k+n&=&n+ (g_{k+1}-g_{k})+(g_{k+2}-g_{k+1})+\cdots+(g_{n_{i}}-g_{n_{i}-1})\\
		&=&n+\frac{g_{k+1}-g_k}{p_k}p_k+\frac{g_{k+2}-g_{k+1}}{p_{k+1}}p_{k+1}+\cdots+\frac{g_{n_{i}}-g_{n_{i}-1}}{p_{n_{i}-1}}p_{n_{i}-1}.
	\end{eqnarray*}
	Since $\frac{g_{n_{i}}-g_{n_{i}-1}} {p_{n_{i}-1}}\in\left\{ 0,1,2, \frac{q_{n_i}}{2}+1, q_{n_i}-1 \right\}$, we have $g_{n_{i}}-g_k+n\in\text{Per}_{p_{n_i}}(\eta)$. For $x\in\pi^{-1}(g)$,  $$-g_k+n\in\text{Per}_{p_{n_i}}(T^{g_{n_i}}\eta)=\text{Per}_{p_{n_i}}(x).$$
	Since $k$ and $n$ are arbitrary ,  $x$ is a Toeplitz sequence. Hence $|\pi^{-1}(g)|=1$.
\end{proof}

\subsection{The verification}

\begin{thm}
	$\pi: (\overline{\mathcal{O}}(\eta),T)\rightarrow (G,\hat{1})$ is almost 1-1 and not saturated.
\end{thm}
\begin{proof}
	Let $A=\{g\in G:~|\pi^{-1}(g)|>1\}$.  By Lemma \ref{Oxtoby seq}, $A=\{g\in G:~|\pi^{-1}(g)|=5\}$.
	By Lemma \ref{meas of reg pts}, $m(A)=1$.  Let $G_0=\{(g_1,g_2,\cdots)\in G: 0\leq g_i< \frac{p_i}{2}\}$ and let $A_0=A\cap G_0$. Then $A_0$ has positive Haar measure. Hence $2A_0-A_0$ contains a nonempty set $U$.
	Next we show that for any $x\in \pi^{-1}(U\setminus A)$, $L_x=\overline{\mathcal{O}}((x,x), T\times T^2)$ is not saturated.\\
	
	Let $g=\pi(x)$. Recall that $\pi(x)\in 2A_0-A_0$. Thus $g=2a-b$ for some $a,b\in A_0$. Suppose that $T^{k_i}\eta\rightarrow x_1\in\pi^{-1}(a)$ and $T^{2k_i}\eta\rightarrow x_2\in\pi^{-1}(b)$.\\
	
	Firstly, assume $x=\eta$. Then $2a=b$ and  $L_{\eta}$ is not saturated by the following claim.\\
	
	\noindent{\bf Claim}. If $x_1(n)=0$ for $n\in \text{Aper}(a)$, then $x_2(n)\neq 2$ for $n\in \text{Aper}(b)$.\\
	\begin{proof}[Proof of Claim]
		Let $r(i)$ be a positive integer such that $\widehat{k_i}$ agrees with $a$ on $[1, r(i)]$. Then we can write $k_i$ in the form of
		$$k_i=a_{r(i)}+s_0p_{r(i)}+s_1p_{r(i)+1}+s_2p_{r(i)+2}+\cdots,$$
		where $s_j\in\{0,1,\cdots, q_{r(i)+j+1}-1\}$.   It follows that $T^{k_i}\eta$ has the same $p_{r(i)}$-skeleton as $T^{a_{r(i)}}\eta$. Thus $T^{k_i}\eta$ has the same $p_{r(i)}$-skeleton as $x_1$.  Set
		$$t_i=\min\left\{j\geq 0: s_j\in\{0,1,2, \frac{q_{r(i)+j+1}}{2}+1, q_{r(i)+j+1}-1\}\right\}.$$
		Then for any $n\in\text{Aper}(x_1)\cap [-a_{r(i)}, p_{r(i)}-a_{r(i)})$, $x_1(n)=0$ and
		$$T^{k_i}\eta(n)=\eta(n+k_i)=
		\begin{cases}
			0,& \text{ if }~s_{t_i}=0,\\
			1,& \text{ if }~s_{t_i}=1,\\
			2,& \text{ if }~s_{t_i}=2,\\
			3,&\text{ if }~s_{t_i}=\frac{q_{r(i)+t_i+1}}{2}+1,\\
			4,&\text{ if }~s_{t_i}=q_{r(i)+t_i+1}-1.
		\end{cases}$$
		
		Since $\widehat{k_i}\rightarrow a$, $r(i)\rightarrow \infty$. By passing to some subsequence, we may assume that $r(i)$ is increasing with respect to $i$.
		Note that $-a_{r(i)}\rightarrow-\infty$ and $p_{r(i)}-a_{r(i)}\rightarrow \infty$, otherwise $a=\widehat{m}$ for some integer $m$ and then $a\notin A$.  Thus for all sufficiently large $i$,    $\text{Aper}(x_1)\cap [-a_{r(i)}, p_{r(i)}-a_{r(i)})\neq \emptyset$.  Note that $[-a_{r(i)}, p_{r(i)}-a_{r(i)})$ is nested with respect to $i$. Therefore, there exists $i_0$ such that for any $i\geq i_0$, $s_{t_i}=0$.
		
		Note that
		\begin{eqnarray*}
			2k_i&=&2a_{r(i)}+2s_0p_{r(i)}+2s_1p_{r(i)+1}+2s_2p_{r(i)+2}+\cdots\\
			&=&b_{r(i)}+s'_0p_{r(i)}+s'_1p_{r(i)+1}+s'_2p_{r(i)+2}+\cdots,
		\end{eqnarray*}
		where $s'_j\in\{0,1,\cdots, q_{r(i)+j+1}-1\}$. Then for each $j$,
		$$s'_j=\begin{cases}
			2s_j \text{ or } 2s_j+1, & \text{ if } s_{j}<\frac{q_{r(i)+j+1}}{2},\\
			2s_j-q_{r(i)+j+1} \text{ or } 2s_{j}-q_{r(i)+j+1}+1, & \text{ if } s_{j}\geq\frac{q_{r(i)+j+1}}{2}.
		\end{cases}$$
		Set $$t'_i=\min\left\{j\geq 0: s'_j\in\{0,1,2, \frac{q_{r(i)+j+1}}{2}+1,q_{r(i)+j+1}-1\}\right\}.$$
		Then for any $n\in\text{Aper}(x_1)\cap [-b_{r(i)}, p_{r(i)}-b_{r(i)})$,
		$$T^{2k_i}\eta(n)=\eta(n+2k_i)=
		\begin{cases}
			0,& \text{ if }~s'_{t'_i}=0,\\
			1,& \text{ if }~s'_{t'_i}=1,\\
			2,& \text{ if }~s'_{t'_i}=2,\\
			3,&\text{ if }~s'_{t'_i}=\frac{q_{r(i)+t'_{i}+1}}{2}+1,\\
			4,&\text{ if }~s'_{t'_i}=q_{r(i)+t'_i+1}-1.
		\end{cases}$$
		
		For any $i\geq i_0$,  since $s_{t_{i}}=0$, we have $s'_{t_i}=0$ or $1$. Thus $t'_{i}\leq t_{i}$.  In the following, we assume that $i\geq i_0$. If $s'_{t'_i}=2$, then $t'_{i}<t_i$ and $s_{t'_{i}}=1$ or $\frac{q_{r(i)+t'_{i}+1}}{2}+1$. This contradicts the choice of $t_{i}$ which is the smallest integer $j$ such that $$s_j\in
		\left\{0,1,2, \frac{q_{r(i)+j+1}}{2}+1,q_{r(i)+j+1}-1\right\}.$$
		
		Therefore, for any $i\geq i_0$ and  $n\in\text{Aper}(x_2)\cap [-b_{r(i)}, p_{r(i)}-b_{r(i)})$, $T^{2k_i}(n)\neq 2$. Hence $x_2(n)\neq 2$ for $n\in \text{Aper}(b)$. This proves the claim.
		
	\end{proof}
	
	Now we are going to deal with the general case. We write $a=g+h$ and $b=g+2h$ for some $h\in G$.    In this case we also claim that if $x_1(n)=0$ for $n\in \text{Aper}(a)$, then $x_2(n)\neq 2$ for $n\in \text{Aper}(b)$.
	
	Let $r(i)$ be a positive integer such that $\widehat{k_i}$ agrees with $h$ on $[1, r(i)]$. Furthermore, we may assume that $r(i)$ is increasing with respect to $i$.  Then we can write $k_i$ in the form of
	$$k_i=h_{r(i)}+s_0p_{r(i)}+s_1p_{r(i)+1}+s_2p_{r(i)+2}+\cdots,$$
	where $s_j\in\{0,1,\cdots, q_{r(i)+j+1}-1\}$.    It follows that $T^{k_i}x$ has the same $p_{r(i)}$-skeleton as $T^{h_{r(i)}}x$. Thus $T^{k_i}x$ has the same $p_{r(i)}$-skeleton as $x_1$.  Then for any $n\in\text{Aper}(x_1)\cap [-a_{r(i)}, p_{r(i)}-a_{r(i)})$, $x_1(n)=0$ and
	$$n+k_i\in [-g_{r(i)}, p_{r(i)}-g_{r(i)})+l_i,$$
	where
	\begin{equation*}
		l_i=\begin{cases}
			k_i-h_{r(i)}, & \text{ if }~a_{r(i)}=g_{r(i)}+h_{(r(i))},\\
			k_{i}-h_{r(i)}+p_{r(i)}, & \text{ if }~a_{r(i)}=g_{r(i)}+h_{(r(i))}-p_{r(i)}.
		\end{cases}
	\end{equation*}
	
	Since $a_{r(i)}, b_{r(i)}< \frac{p_{r(i)}}{2}$, if $a_{r(i)}=g_{r(i)}+h_{(r(i))} $ then $b_{r(i)}=g_{r(i)}+2h_{(r(i))}$, and if $a_{r(i)}=g_{r(i)}+h_{(r(i))}-p_{r(i)} $ then $b_{r(i)}=g_{r(i)}+2h_{(r(i))}-p_{r(i)}$. Thus for any $m\in\text{Aper}(x_2)\cap [-b_{r(i)}, p_{r(i)}-b_{r(i)})$,
	$$m+2k_i\in [-g_{r(i)}, p_{r(i)}-g_{r(i)})+l'_i,$$
	where
	\begin{equation*}
		l'_i=\begin{cases}
			2(k_i-h_{r(i)}), & \text{ if }~a_{r(i)}=g_{r(i)}+h_{(r(i))},\\
			2(k_{i}-h_{r(i)})+p_{r(i)}, & \text{ if }~a_{r(i)}=g_{r(i)}+h_{(r(i))}-p_{r(i)}.
		\end{cases}
	\end{equation*}
	
	Write $l_i=s_0p_{r(i)}+s_1p_{r(i)+1}+s_2p_{r(i)+2}+\cdots$,  where $s_j\in\{0,1,\cdots, q_{r(i)+j+1}-1\}$.     Set
	$$t_i=\min\left\{j\geq 0: s_j\in\{0,1,2, \frac{q_{r(i)+j+1}}{2}+1, q_{r(i)+j+1}-1\}\right\}.$$
	Then for any $n\in\text{Aper}(x_1)\cap [-a_{r(i)}, p_{r(i)}-a_{r(i)})$, $x_1(n)=0$ and
	$$T^{k_i}x(n)=x(n+k_i)=
	\begin{cases}
		0,& \text{ if }~s_{t_i}=0,\\
		1,& \text{ if }~s_{t_i}=1,\\
		2,& \text{ if }~s_{t_i}=2,\\
		3,&\text{ if }~s_{t_i}=\frac{q_{r(i)+t_i+1}}{2}+1,\\
		4,&\text{ if }~s_{t_i}=q_{r(i)+t_i+1}-1.
	\end{cases}$$
	
	\noindent {\bf Case 1}. $a_{r(i)}=g_{r(i)}+h_{(r(i))}$.
	
	Then we have $l'_i=2l_i$. Similar to the case that $x=\eta$, we have $T^{2k_i}x(n)\neq 2$ for any $n\in\text{Aper}(x_2)\cap [-b_{r(i)}, p_{r(i)}-b_{r(i)})$.
	
	\noindent{\bf Case 2}. $a_{r(i)}=g_{r(i)}+h_{(r(i))}-p_{r(i)}$.
	
	In this case, we have $l'_i=2l_i-p_i$.  We write $2l_i=s'_0p_{r(i)}+s'_1p_{r(i)+1}+s'_2p_{r(i)+2}+\cdots$,  where $s'_j\in\{0,1,\cdots, q_{r(i)+j+1}-1\}$.
	
	\noindent {\bf Case 2a}. $s'_0=0$.
	
	Then $l'_i=(q_{r(i)+1}-1)p_{r(i)}+sp_{r(i)+1}$ for some nonnegative integer $s$. For any $n\in\text{Aper}(x_2)\cap [-b_{r(i)}, p_{r(i)}-b_{r(i)})$, we have $T^{2k_i}x(n)=0$.
	
	\noindent {\bf Case 2b}. $s'_0>0$.
	
	Then $l'_i=(s'_0-1)p_{r(i)}+s'_1p_{r(i)+1}+s'_2p_{r(i)+2}+\cdots$.  Set $$t'_i=\min\left\{j\geq 0: s'_j\in\{0,1,2, \frac{q_{r(i)+j+1}}{2}+1,q_{r(i)+j+1}-1\}\right\}.$$
	
	Note that $s'_0=2s_0$ or $2s_0-p_{r(i)}$ and for each $j>0$,
	$$s'_j=\begin{cases}
		2s_j \text{ or } 2s_j+1, & \text{ if } s_{j}<\frac{q_{r(i)+j+1}}{2},\\
		2s_j-q_{r(i)+j+1} \text{ or } 2s_{j}-q_{r(i)+j+1}+1, & \text{ if } s_{j}\geq\frac{q_{r(i)+j+1}}{2}.
	\end{cases}$$
	
	If $t'_i=0$ and $s'_0-1=2$, then $s'_0=3$. This is impossible since $p_{r(i)}$ is an even number. If $t'_i>0$, then $s'_{t'_i}\neq 2$ by the same reason as the discussion in case of $x=\eta$. In any case, we have $T^{2k_i}x(n)\neq 2$ for any $n\in\text{Aper}(x_2)\cap [-b_{r(i)}, p_{r(i)}-b_{r(i)})$.
	
	By the above discussion, we have $x_2(n)\neq 2$ for $n\in\text{Aper}(b)$. Hence $L_{x}$ is not saturated.
	
\end{proof}

\noindent {\bf Acknowledgements:} We thank E. Glasner for asking the question if there is a metric non-saturation example.
We also thank W. Huang and S. Shao for useful discussions.


\end{document}